\documentclass[12pt,reqno]{amsart}

\usepackage[margin=1in]{geometry}  % set the margins to 1in on all sides
\usepackage{graphicx}              % to include figures
\usepackage{amsmath}               % great math stuff
\usepackage{amsfonts}              % for blackboard bold, etc
\usepackage{amsthm}                % better theorem environments
\usepackage{amssymb}
\usepackage{appendix}

\newtheorem{thm}{Theorem}[section]
\newtheorem{lem}[thm]{Lemma}
\newtheorem{prop}[thm]{Proposition}

\newtheorem{conj}[thm]{Conjecture}

\newcommand{\bd}[1]{\mathbf{#1}}  % for bolding symbols
\newcommand{\RR}{\mathbb{R}}      % for Real numbers
\newcommand{\ZZ}{\mathbb{Z}}      % for Integers
\newcommand{\CC}{\mathbb{C}}

\newcommand{\PC}{\mathcal{P}}
\newcommand{\Aa}{\mathcal{A}}
\newcommand{\mat}[1]{\left(\begin{matrix} #1 \end{matrix} \right)}  % for matrix
\newcommand{\srt}{\sqrt{2}} % for square root 2
\newcommand{\al}[1]{\begin{align}#1\end{align}}
\newcommand{\aln}[1]{\begin{align*}#1\end{align*}}
\newcommand{\ff}{\mathfrak{f}}
\newcommand{\fff}{\tilde{\mathfrak{f}}}
\newcommand{\FF}{\mathfrak{F}}
\newcommand{\rn}{\mathcal{R}_N(n)}
\newcommand{\hrn}{\widehat{\mathcal{R}}_N}
\newcommand{\rnu}{\mathcal{R}_N^U}
\newcommand{\hrnu}{\widehat{\mathcal{R}}_N^U}
\newcommand{\tf}{\mathfrak{t}}
\newcommand{\mn}{\mathcal{M}_N}
\newcommand{\en}{\mathcal{E}_N}
\newcommand{\mnn}{\mathcal{M}_N(n)}
\newcommand{\enn}{\mathcal{E}_N(n)}
\newcommand{\enu}{\mathcal{E}_N^U}
\newcommand{\mnu}{\mathcal{M}_N^{U}}
\newcommand{\sk}[1]{\substack{#1}}
\newcommand{\Tf}{\mathfrak{T}}
\newcommand{\mci}{\mathcal{I}}
\newcommand{\ruf}{\mathcal{R}_{u,\ff}}
\newcommand{\infint}{\int_{-\infty}^{\infty}}
\newcommand{\dinfint}{\int_{-\infty}^{\infty}\int_{-\infty}^{\infty}}

\newcommand{\II}{\mathcal{I}}
\newcommand{\JJ}{\mathcal{J}}
\newcommand{\SM}{\mathcal{S}}
\newcommand{\stwist}{\mathcal{S}(q,q_0,u_0,\xi,\zeta,\ff,\xi^{'},\zeta^{'},\ff^{'})}
\newcommand{\bff}{b_{\ff}}
\newcommand{\bfp}{b_{\ff^{'}}}
\newcommand{\lle}{\ll_{\epsilon}}
\newcommand{\nep}{N^{\epsilon}}

\newcommand{\iqs}{\mathcal{I}_Q^{(\neq,\leq)}}
\newcommand{\ffp}{\mathfrak{f}^{'}}
\newcommand{\KK}{\mathcal{K}}
\newcommand{\llet}{\ll_{\eta}}

\begin{document}

\nocite{*}

\title{On the Local-Global Principle for Integral Apollonian 3-Circle Packings}

\author{Xin Zhang}

\maketitle

\begin{abstract} 
In this paper we study the integral properties of Apollonian-3 circle packings, which are variants of the standard Apollonian circle packings.  Specifically, we study the reduction theory, formulate a local-global conjecture, and prove a density one version of this conjecture.  Along the way, we prove a uniform spectral gap for the congruence towers of the symmetry group.  
\end{abstract}

\section{Introduction}

Apollonian circle packings are well-known planar fractal sets.  Starting with three mutually tangent circles, we inscribe one circle into each curvilinear triangle.  Repeat this process ad infinitum and we get an Apollonian circle packing.  Soddy first observed the existence of some Apollonian packings with all circles having integer curvatures, and we call these packings \emph{integral}.  The systematic study of the integers from such packings was initiated by Graham, Lagarias, Mallows, Wilks, and Yan \cite{GLMWY03} \cite{GLMWY05}.  We first briefly review what is known for integral Apollonian packings.  Fix an integral Apollonian packing $\mathcal{P}$, and let $\mathcal{K}$ be the set of curvatures from $\mathcal{P}$.  Without loss of generality we can assume $\mathcal{P}$ is \emph{primitive} (i.e. the $gcd$ of $\mathcal{K}$ is 1).  We say an integer $n$ is $admissible$ if it passes all local obstructions (i.e. for any $q$, we can find $\kappa\in\KK$ such that $n\equiv\kappa$ (mod $q$)).   Finally, let $\Gamma$ be the orientation-preserving symmetry group acting on $\mathcal{P}$, which is an infinite co-volume Kleinian group.  We have:\\

(1) \emph{The reduction theorem}: Fuchs in her thesis \cite{Fu10} proved that an integer is admissible if and only if it passes the local obstruction at 24.\\

(2) {\emph{The local-global conjecture}}:  Graham, Lagarias, Mallows, Wilks, Yan \cite{GLMWY03} conjectured that every sufficiently large admissible integer is actually a curvature.\\

(3)\emph{A congruence subgroup}: Sarnak \cite{SaLa} observed that there is a real congruence subgroup lying in $\Gamma$.  As a consequence, some curvatures can be represented by certain shifted quadratic forms.\\

(4)\emph{The congruence towers of $\Gamma$ has a spectral gap }(See Page 3 for definition): This fact was proved by Varj\"{u} in the appendix of \cite{BK13}, using Theorem 1.2 of \cite{BGS11}.  \\

(5)\emph{A density one theorem}:  Building on the works of Sarnak \cite{SaLa}, Fuchs \cite{Fu10}, and Fuchs-Bourgain \cite{BF12}, Bourgain and Kontorovich \cite{BK13} proved that almost every admissible integer is a curvature, which is a step towards the local-global conjecture.\\

 \begin{figure}[htbp] %  figure placement: here, top, bottom, or page
    \centering
    \includegraphics[width=2.5in]{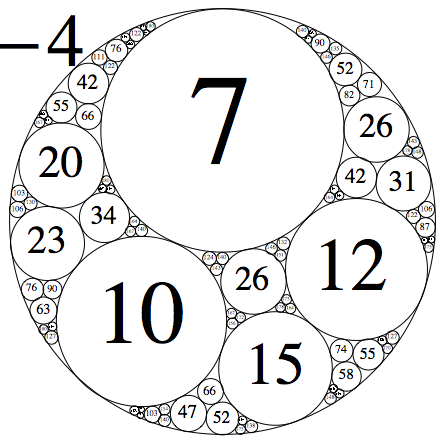} 
   \caption{An integral Apollonian-3 circle packing}
    \label{fig1}
 \end{figure}

In this paper we generalize the above results to the type of circle packings illustrated in Figure \ref{fig1}.   To construct such a packing,  we begin with three mutually tangent circles.  We iteratively inscribe three circles into curvilinear triangles, and obtain a circle packing, which we call an \emph{Apollonian 3-circle packing}, or Apollonian 3-packing. (By comparison,  if we inscribe one circle in each gap,  we obtain a standard Apollonian packing.)  As shown in Figure 1, there also exist integral Apollonian-3 packings.  This was first observed by Guettler-Mallows \cite{GM08}.\\

We carry over the notations $\PC, \KK,\Gamma$ to our Apollonian-3 setting.  We fix a primitive Apollonian-3 packing $\mathcal{P}$, let $\KK$ be the set of curvatures from $\mathcal{P}$, and $\Gamma$ be the orientation-preserving symmetry group acting on $\PC$.   We first state a reduction theorem for $\PC$.    
\begin{thm}{(Reduction Theorem)}\label{reduction} An integer $n$ is admissible by $\PC$ if and only if it passes the local obstruction at 8.
\end{thm}
Let $A_{\PC}$ be the set of admissible integers of $\PC$.  In the case of Figure 1, 
$$A_{\PC}=\{n\in\ZZ|n\equiv 2,4,7(\text{mod }8)\}.$$

A general result from Weisfeiler \cite{We84} implies the existence of a number $Q$ which completely determines the local obstruction.  However in practice it's a hard problem to determine $Q$.  In our case $Q=8$.  Technically, we will prove the following lemma, which directly implies Theorem \ref{reduction}.   Let $\KK_d$ be the reduction of $\mathcal{K}$ (mod $d$), and ${\rho_{p^m}}$ be the natural projection from $\ZZ/p^{m+1}\ZZ$ to $\ZZ/p^{m}\ZZ$.  Write $d=\prod_{i}p_i^{n_i}$, then we have

\begin{lem}\label{local}\text{}\\
{\rm(1)} $\mathcal{K}_q\cong \prod_{i}\mathcal{K}_{p_i^{n_i}}$,\\
{\rm(2)} $\mathcal{K}_{p^{m}}=\ZZ/p^{m}\ZZ$ for $p\geq 3$ and $m\geq 0$,\\
{\rm(3)} $\rho_{2^{m+1}}^{-1}(\mathcal{K}_{2^{m}})=\mathcal{K}_{2^{m+1}}$ for $p=2$ and $m\geq 3$.
\end{lem} 

Based on Theorem \ref{reduction}, we formulate the following local-global conjecture:
\begin{conj}{(Local-global Conjecture)} Every sufficiently large admissible integer from $\PC$ is a curvature.  Or equivalently,
$$\#\{n\in \KK|n\leq N \}=\#\{n\in A_{\PC}|0<n\leq N \}+O(1).$$
\end{conj}

However, it seems that the current technology is not enough to deal with this conjecture.  Instead,  we prove a density one theorem: 
\begin{thm}{(Density One Theorem)}\label{mainthm}
There exists $\eta>0$ such that $$\#\{n\in \KK|n\leq N \}=\#\{n\in A_{\PC}|0<n\leq N \}+O(N^{1-\eta}).$$
.
\end{thm}

To deduce Theorems \ref{reduction} and \ref{mainthm}, we need to study the symmetry group $\tilde{\Gamma}$, or more conveniently its orientation preserving subgroup $\Gamma$.  The group $\tilde{\Gamma}\subset\text{Isom}(\mathbb{H}^3)$  is generated by eight reflections corresponding to eight mutually disjoint hemispheres, and our Apollonian-3 packing can be realized as the limit set of a point orbit under $\tilde{\Gamma}$ (see Figure \ref{fig:fundamentaldomain}).  Therefore $\Gamma$ is geometric finite.  It is clear that $\Gamma\backslash \mathbb{H}^3$ has infinite volume, so $\Gamma$ is a thin subgroup of $SL(2,\CC)$.  The local structure of $\Gamma$ will lead to Theorem \ref{reduction}.  Here we exploit a crucial fact that $\Gamma$ contains a real congruence subgroup, which is the analogue of Sarnak's observation for the Apollonian group \cite{SaLa}.  This congruence subgroup also implies that some curvatures can be represented by certain shifted binary quadratic forms (See Theorem \ref{shiftedquadraticform}), which is a key starting point for proving Theorem \ref{mainthm}.\\

    Another crucial ingredient for Theorem \ref{mainthm} is a (geometric) \emph{spectral gap} for $\Gamma$, as we explain now.  For any positive integer $q$, Let $\Gamma(q)$ be the principle congruence subgroup of $\Gamma$ at $q$ (i.e.  $\Gamma(q)=\{\gamma\in\Gamma |\gamma\equiv I\text{(mod }q)\})$.   Let $\Delta$ be the hyperbolic Laplacian operator associated to the metric $ds^2=\frac{dx^2+dy^2+dz^2}{z^2}$ on $\mathbb{H}^3$:
$$\Delta=-z^2(\frac{\partial^2}{\partial x^2}+\frac{\partial^2}{\partial y^2}+\frac{\partial^2}{\partial z^2})+z\frac{\partial}{\partial z}$$

The operator $\Delta$ is symmetric and positive definite on $L^2(\Gamma(q)\backslash \mathbb{H}^3)$ with the standard inner product.  From Larman \cite{La67} we know that the Hausdorff dimension $\delta$ of our packing $\PC$ is $>1$.  Hence Patterson-Sullivan theory \cite{Pa76}\cite{Su84}, together with Lax-Phillips\cite{LP82} tell us that for each $q$, there are only finitely many exceptional eigenvalue for $\Delta$ acting on $L^2(\Gamma(q)\backslash \mathbb{H}_3)$, and the base (smallest) eigenvalue $\lambda_0(q)$ of $\Delta$ on $L^2(\Gamma(q)\backslash \mathbb{H}^3)$ is equal to $\delta(2-\delta)$. \\

However, a priori the second smallest eigenvalue $\lambda_1(q)$ might get arbitrarily close to $\lambda_0(q)$.  But in the case of $\Gamma$, this phenomenon does not happen:
 \begin{figure}[htbp] %  figure placement: here, top, bottom, or page
    \centering
    \includegraphics[width=4in]{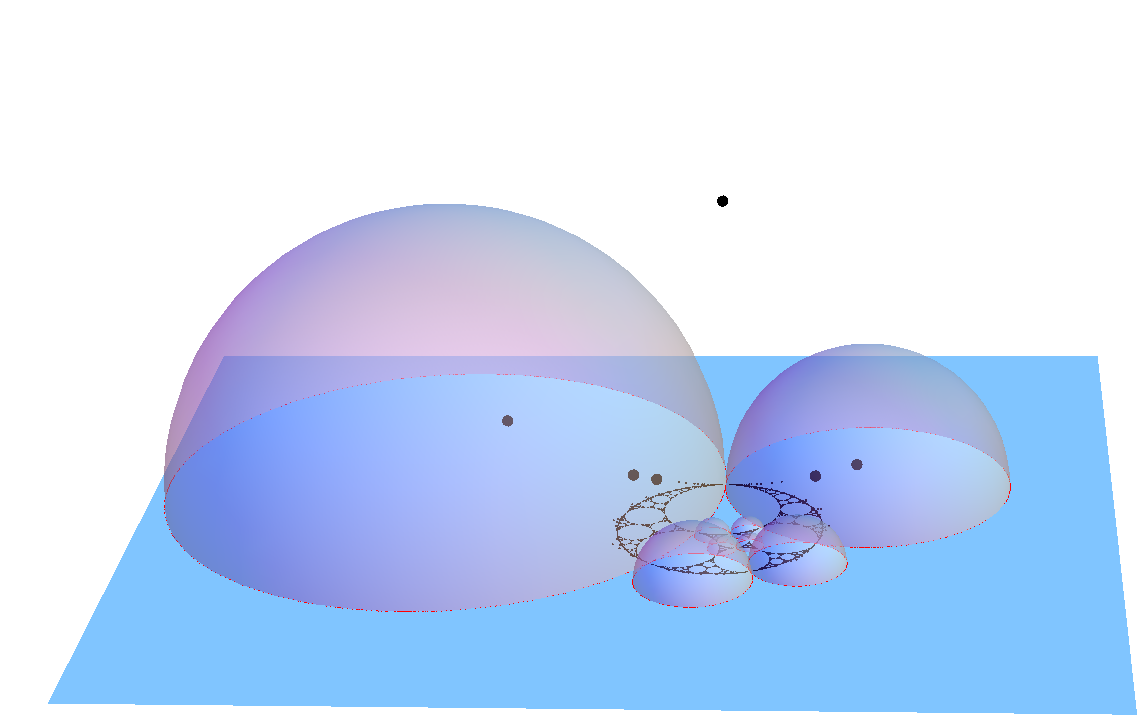} 
   \caption{The fundamental domain for $\tilde{\Gamma}$ and the orbit of an point under $\tilde{\Gamma}$}
    \label{fig:fundamentaldomain}
 \end{figure}
\begin{thm}{(Spectral Gap)}\label{spec}
 There exists $\delta_0>0$ such that for all $q$, 
$$\lambda_1(q)-\lambda_0(q)\geq \delta_0$$
\end{thm}

For the modular group $SL(2,\ZZ)$, the celebrated Selberg $\frac{3}{16}$ Theorem says that $\delta_0\geq \frac{3}{16}$.  For an arbitrary finitely generated subgroup of $SL(2,\ZZ)$, a spectral gap when $q$ is ranging over square free numbers was obtained by Bourgain-Gamburd-Sarnak \cite{BGS11}.  Recently this result was extended to much more general groups by    
Golsefidy-Varj\'u\cite{GV12}, again over squarefree numbers .  But for our need, we need to require $q$ to exhaust all integers. \\

We then follow the strategy in \cite{BK13} to prove Theorem \ref{mainthm}.  The main approach is the Hardy-Littlewood circle method.  The spectral gap given in Theorem \ref{spec},  together with the bisector counting result from Vinogradov \cite{Vi13},  allows us to do various (thin) lattice point counting restricted to certain regions of $SL(2,\CC)$, effectively and with uniform rates over the congruence towers $\Gamma(q)$ and their cosets.  All these are encoded in Lemma 5.2, Lemma 5.3 and Lemma 5.4 from Bourgain and Kontorovich's work on Apollonian packings \cite{BK13}.  These Lemmas can be modified word by word to fit our setting.  Another ingredient which appears in the minor arc analysis is the elementary $\frac{3}{4}$ bound for the Kloosterman sums.   
 \\

$\bd{Plan\text{ } for\text{ }  the\text{ }  paper}:$ In \S 2 we discuss the local properties of $\KK$,  these properties are revealed by $\Gamma$ and its subgroups.  Theroems \ref{reduction} and \ref{spec} are proved at the end of this section.   The main goal of \S 3 is to prove Theorem \ref{mainthm}.  In \S 3.1 we introduce the main exponential sum and give an outline of the proof of Theorem \ref{mainthm}.  In \S 3.2 we analyze the major arcs, and from \S 3.3 to \S 3.5 we give bounds for three parts of the minor-arc integrals. Finally in \S 3.6 we conclude our proof.\\

$\bd{Notation}$:  We adopt the following standard notations.  We write $e^{2\pi i x}$ as $e(x)$, and $e^{\frac{2\pi i x}{q}}$ as $e_q(x)$.  The relation $f\ll g$ means that $f=O(g)$, and $f\asymp g$ means $f\ll g$ and $g\ll f$. The Greek letter $\epsilon$ denotes an arbitrary small positive number,  and $\eta$ denotes a small positive number which appears in several contexts.  We assume that each time when $\eta$ appears,  we let $\eta$ not only satisfy the current claim,  but also satisfy the claims in all previous contexts. The symbols $p$ and $p_i$ always denote a prime. The relation $p^j|| n$ means $p^j| n$ and $p^{j+1}\nmid n$.  The expression $\sum_{r(q)}^{'}$ means sum over all $r(\text{mod }q)$ where $(r,q)=1$.  For a finite set $Z$,  its cardinality is denoted by $|Z|$ or $\#Z$.  For an algebraic group $\Gamma$ (or $\Aa$, $\tilde{\Aa}$) over $\ZZ$, $\Gamma(q)$ (or $\Aa(q)$, $\tilde{\Aa}(q)$) denotes its principle congruence subgroup of level $q$.  Without further mentioning,  all the implied constants depend at most on the given packing.  
\section{Local Property}
 \subsection{Apollonian 3-Group and Its Subgroups} % Section 2.1
We start with three mutually tangent circles $C_1,C_2,C_3$ (suppose $C_1$ is bounding the other two).  In each of the two gaps formed by these three circles, there's a unique way to inscribe three more circles, in a way that each of these six circles is tangent to four other circles and disjoint to the last one.  Let's say $C_{1^{'}},C_{2^{'}},C_{3^{'}}$ is one such inscription (see Figure \ref{fig:example}). It is known that their curvatures $\kappa_1,\kappa_2,\kappa_3,\kappa_{1^{'}},\kappa_{2^{'}},\kappa_{3^{'}}$ satisfy the following algebraic relations \cite{GM08}:
\begin{align}
&\kappa_1+\kappa_{1^{'}}=\kappa_2+\kappa_{2^{'}}=\kappa_3+\kappa_{3^{'}}:=2w\label{sum}\\
&Q(\kappa_1,\kappa_2,\kappa_3,w)=w^2-2w(\kappa_1+\kappa_2+\kappa_3)+\kappa_1^2+\kappa_2^2+\kappa_3^2=0\label{quadratic}
\end{align}

The M\"obits inversion via the dual circle of $C_1,C_2,C_3$ takes $C_{1^{'}},C_{2^{'}},C_{3^{'}}$ to three other circles $C_{1^{''}},C_{2^{''}},C_{3^{''}}$, which gives the other way of inscribing.  There are two solutions for $w$ in \eqref{quadratic},  which correspond exactly to two ways of filling. 
 \begin{figure}[htbp] %  figure placement: here, top, bottom, or page
    \centering
    \includegraphics[width=1.7in]{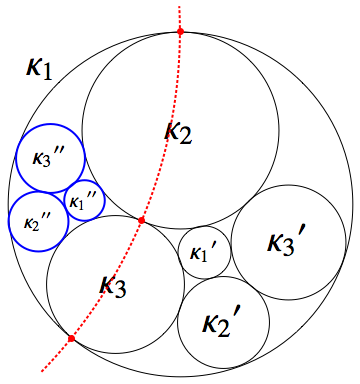} 
   \caption{Reflection via the dual circle of $C_1,C_2,C_3$ }
    \label{fig:example}
 \end{figure}

We associate a quadruple $\bd{r}=\left\langle\kappa_1,\kappa_2,\kappa_3, w\right\rangle^{T}$ to the six circles $C_1,C_2,C_3,C_{1^{'}},C_{2^{'}},C_{3^{'}}$, which we call the \emph{root circles}.  There are eight gaps formed by circular triangles.  Each gap corresponds to one M\"obius inversion,  which takes three of the six root circles to three new circles and fixes the rest three.  We associate a vector $\bd{v}=\left\langle x,y,z,w^{'}\right\rangle^{T}$ to this new collection of six circles, where $x,y,z$ are the curvatures of the circles which are the images of $C_1,C_2,C_3$ under the reflection, and $w^{'}$ is the sum of any pair of disjoint circles from this new collection, as $w$ in \eqref{sum}.  From \eqref{sum} and \eqref{quadratic} it follows that $x,y,z,w^{'}$ has linear dependance on $\kappa_1,\kappa_2,\kappa_3, w$.   Eight gaps correspond to eight linear transformations which take $\bd{r}$ to $\bd{v}$:  

 \begin{align}\nonumber
 &S_{123} = \left(
    \begin{matrix}
      1 & 0 & 0 & 0 \\
      0 & 1 & 0 & 0 \\
      0 & 0 & 1 & 0 \\
      2 & 2 & 2 & -1
    \end{matrix}
    \right),
     &S_{1^{'}23} &= \left(
    \begin{matrix}
      -3 & 4 & 4 & 4 \\
      0 & 1 & 0 & 0 \\
      0 & 0 & 1 & 0 \\
      -2 & 2 & 2 & 3
    \end{matrix}
    \right),\\
    \nonumber
     &S_{12^{'}3} = \left(
    \begin{matrix}
      1 & 0 & 0 & 0 \\
      4 & -3 & 4 & 4 \\
      0 & 0 & 1 & 0 \\
      2 & -2 & 2 & 3
    \end{matrix}
    \right),
     &S_{123^{'}} &= \left(
    \begin{matrix}
      1 & 0 & 0 & 0 \\
      0 & 1 & 0 & 0 \\
      4 & 4 & -3 & 4 \\
      2 & 2 & -2 & 3
    \end{matrix}
    \right),\\
    \nonumber
    &S_{1^{'}2^{'}3} = \left(
    \begin{matrix}
      -3 & -4 & 4 & 12 \\
      -4 & -3 & 4 & 12 \\
      0 & 0 & 1 & 0 \\
      -2 & -2 & 2 & 7
    \end{matrix}
    \right),
    &S_{1^{'}23^{'}} &= \left(
    \begin{matrix}
      -3 & 4 & -4 & 12 \\
      0 & 1 & 0 & 0 \\
      -4 & 4 & -3 & 12 \\
      -2 & 2 & -2 & 7
    \end{matrix}
    \right),\\
    &S_{12^{'}3^{'}} = \left(
    \begin{matrix}
      1 & 0 & 0 & 0 \\
      4 & -3 & -4 & 12 \\
      4 & -4 & -3 & 12 \\
      2 & -2 & -2 & 7
    \end{matrix}
    \right), 
    &S_{1^{'}2^{'}3^{'}} &= \left(
    \begin{matrix}
      -3 & -4 & -4 & 20 \\
      -4 & -3 & -4 & 20 \\
      -4 & -4 & -3 & 20 \\
      -2 & -2 & -2 & 11
    \end{matrix}
    \right).
     \end{align}
 The subtitles of the above notations keep track of the circles forming the triangular gap.  For example,  $S_{1^{'}2^{'}3} $ denotes the reflection via the dual circle of $C_{1^{'}},C_{2^{'}},C_3$.   The group generated by these eight matrices is called Apollonian 3-group, denoted by $\tilde{\mathcal{A}}$:
 \begin{align}\label{a3}\tilde{\mathcal{A}}=\langle S_{123},S_{1^{'}23},S_{12^{'}3},S_{123^{'}},S_{1^{'}2^{'}3},S_{1^{'}2^{'}3},S_{1^{'}23^{'}},S_{1^{'}2^{'}3^{'}}\rangle\end{align}
 
   Then we have 
 \begin{align}\label{KK} \KK=\{\langle\bd{e}_i, \tilde{\mathcal{A}}\cdot\bd{r}\rangle|i=1,2,3\}\cup\{\langle\bd{e}_i, \tilde{\mathcal{A}}\cdot\bd{r}^{'}\rangle|i=1,2,3\}
 \end{align}
 where $\bd{r^{'}}=\langle\kappa_{1^{'}},\kappa_{2^{'}},\kappa_{3^{'}},w\rangle$.  It then follows that if the initial six circles have integral curvatures,  then $\mathcal{P}$ is integral.

 In light of \eqref{K},  we reduce studying $\KK$ to studying the group $\tilde{\mathcal{A}}$ which acts on some quadruples containing full information of $\KK$.  $\Aa$ is a Coxeter group with the only relations $$S_{123}^2=S_{1^{'}23}^2=\hdots=I.$$  It preserves the quadratic 3-1 form $Q$,  so $\tilde{Aa}\subseteq{O_Q(\ZZ)}$.  Furthermore, we pass to its orientation-preserving subgroup ${\Aa}=\Aa\cap SO_Q(\ZZ)$, which is an index-2 subgroup of ${\tilde{\Aa}}$ and a free group generated by 
\begin{align}\label{generator}
S_{123}S_{1^{'}23},S_{123}S_{12^{'}3},S_{123}S_{123^{'}},S_{123}S_{1^{'}2^{'}3},S_{123}S_{1^{'}23^{'}},S_{123}S_{12^{'}3^{'}},S_{123}S_{1^{'}2^{'}3^{'}}.\end{align}\\
From \eqref{KK} we also have 
 \begin{align}\label{K} \KK=\{\langle\bd{e}_i, {\mathcal{A}}\cdot\bd{r}\rangle|i=1,2,3\}\cup\{\langle\bd{e}_i, \mathcal{A}\cdot\bd{r}^{'}\rangle|i=1,2,3\}
 \end{align}
 This is because if a word from $\tilde{\Aa}$ consists of odd number of reflections, we can always pre-add $S_{123}$ (or $S_{1^{'}2^{'}3^{'}}$) without changing $\bd{r}$ (or $\bd{r}^{'}$).  The augmented word is even, thus lies in $\Aa$.\\

Recall the spin homomorphism $\rho_0:SL(2,\CC)\longrightarrow SO_{{Q_0}}$, where $\tilde{Q_0}(x,y,z,t)=t^2-x^2-y^2-z^2$ is the standard $3-1$ form (see \cite{EGJ98}):
\begin{equation}
\rho_0\left(
\left(
\begin{array}{ccc}
  a & b     \\
  c & d     
\end{array}
\right)
\right)=
\left(
\begin{matrix}
 \Re(a\bar{d}+b\bar{c}) &\Im(a\bar{d}-b\bar{c})  &\Re(-a\bar{c}+b\bar{d}) &\Re(a\bar{c}+b\bar{d}) \\
 \Im(-a\bar{d}-b\bar{c})&\Re(a\bar{d}-b\bar{c}) & \Im(a\bar{c}-b\bar{d})& \Im(-a\bar{c}-b\bar{d})\\
 \Re(-a\bar{b}+c\bar{d})& \Im(-a\bar{b}+c\bar{d}) &\frac{|a|^2-|b|^2-|c|^2+|d|^2}{2}&\frac{-|a|^2-|b|^2+|c|^2+|d|^2}{2}\\
 \Re(a\bar{b}+c\bar{d})&\Im(a\bar{b}+c\bar{d})&\frac{-|a|^2+|b|^2-|c|^2+|d|^2}{2}&\frac{|a|^2+|b|^2+|c|^2+|d|^2}{2}
 \end{matrix}
\right)
\end{equation}

The isomorphism between $SO_{{Q_0}}$ and $SO_{Q}$ is given by 
\begin{align*}
{A}\longrightarrow J^{-1}{A}J,
\end{align*}
where 
\begin{align*}
J=\left(\begin{matrix}1&0&0&-1\\0&1&0&-1\\0&0&1&-1\\0&0&0&\sqrt{2}\end{matrix}\right)
\end{align*}

The spin homomorphism that we use is  $\rho$, defined from $SL(2,\CC)$ to $SO_{\tilde{Q}}$ as

\begin{equation}
\rho(\gamma)=J^{-1}\rho_0\left(\left(\begin{matrix}1+i&-\sqrt{2}\\\sqrt{2}&1+i\end{matrix}\right)\gamma\left(\begin{matrix}1+i&-\sqrt{2}\\\sqrt{2}&1+i\end{matrix}\right)^{-1}\right)J
\end{equation}

The good thing about conjugating $\gamma$ with $\left(\begin{matrix}1+i&-\sqrt{2}\\\sqrt{2}&1+i\end{matrix}\right)$ is that the preimage of the generators in \eqref{generator} is 

\begin{align}\nonumber
& M_1=\mat{1&2\\-2&-3},M_2=\mat{1-2\sqrt{2}i&2\\2+4\sqrt{2}i&-3+2\sqrt{2}i},M_3=\mat{1&0\\-4&1},M_4=\mat{-1+2\sqrt{2}i&-4\\-4\sqrt{2}i&7-2\sqrt{2}i}\\
& M_5=\mat{-1&2\\2&-5},M_6=\mat{1+2\sqrt{2}i&-2\\-6-4\sqrt{2}i&5-2\sqrt{2}i},
M_7=\mat{-1-2\sqrt{2}i&4\\4+4\sqrt{2}i&-9+2\sqrt{2}i},
\end{align}
which all lie in $SL(2,\ZZ [\srt i])$, and we let $\Gamma=\langle M_1, M_2, M_3, M_4,M_5,M_6,M_7\rangle$. \\

If we write $a=a_1+a_2\bd{i}, b=b_1+b_2\bd{i}, c=c_1+c_2\bd{i}, d=d_1+d_2\bd{i}$, one can verify (with the aid of computer) that $\rho$ maps the matrix $ \left(\begin{array}{ccc}
  a & b     \\
  c & d     
\end{array}
\right)$ to a $4\times4$ matrix, each entry of which is a homogenous quadratic polynomial of $a_1,a_2,b_1,b_2,c_1,c_2,d_1,d_2$, with half-integer coefficients.  Therefore, $\rho$ can descend to a homomorphism from 
$\Gamma/\Gamma(q)$ to $A/A(q)$ for any $q$ that does not contain a power of 2.  \\

The group $\Gamma$ contains a real subgroup $\Gamma_{C_3}=\langle M_1,M_3,M_5\rangle$.  Geometrically,  $\Gamma_{C_3}$ fixes the circle $C_3$.  It turns out that $\Gamma_{C_3}$ is a congruence subgroup:

 \begin{prop}The group $\Gamma$ is a congruence subgroup of level 4. Explicitly, 
\al{\label{congruencesubgroup4}\Gamma_{C_3}=\left\{\mat{a&b\\c&d}\in SL(2,\ZZ)|a\equiv d\equiv 1(\rm{mod}\hspace{1.5mm}2),b\equiv c\equiv0 \hspace{1mm}\rm{or} \hspace{1mm}2 (\rm{mod}\hspace{1.5mm}4)\right\}} 
 \end{prop}
 
  \begin{proof}
We notice that the $\subseteq$ direction is straightforward, then we can prove the proposition by explicitly constructing the fundamental domain (See Figure \ref{fundamentaldomaincongruence}). Indeed once we show that the fundamental domain of $\Gamma_{C_3}$ is as shown in Figure  \ref{fundamentaldomaincongruence}, we can compute the covolume of $\Gamma_{C_3}$ to be $8\pi$, which coincides with the covolume of the group described by the righthand side of  \eqref{congruencesubgroup4}, thus the proposition is established. \\

First we replace the generators $M_1,M_3,M_5$ of $\Gamma_{C_3}$ by three parabolic generators $M_1,M_3^{-1}=\left(\begin{array}{ccc}
  1 & 0     \\
  4 & 1     
\end{array}
\right),M_3^{-1}M_5=\left(\begin{array}{ccc}
  -1 & 2     \\
  -2 & 3     
\end{array}
\right)$ which fix -1,0,1 respectively. We denote the corresponding parabolic subgroups by $B_{-1},B_{0},B_1$.  We have $M_1(\infty)=-\frac{1}{2},M_3^{-1}(-\frac{1}{2})=\frac{1}{2}$ and $M_3^{-1}M_5(\frac{1}{2})=\infty$. It turns out that the open region $\mathcal{F}_{C_3}$ bounded by the closed loop $\infty\rightarrow -1\rightarrow -\frac{1}{2}\rightarrow0\rightarrow\frac{1}{2}\rightarrow1\rightarrow\infty$ is the fundamental domain for $\Gamma_{C_3}$.  

Associate the open regions $\rm{I},\rm{II},\rm{III}$ (See Figure \ref{fundamentaldomaincongruence}) to $B_{-1},B_0,B_1$, then $B_{-1}$ maps $\rm{II},\rm{III}$ to $\rm{I}$,  $B_0$ maps $\rm{I},\rm{III}$ to $\rm{II}$ and $B_1$ maps $\rm{I},\rm{II}$ to $\rm{III}$.  We can apply the Pingpong Lemma to show that $\Gamma$ is freely generated  by these three elements.  To show that $\mathcal{F}_{C_3}$ is a fundamental domain, one needs to show \\

(i)$\gamma(\mathcal{F}_{C_3})\cap\mathcal{F}_{C_3}=\emptyset$ if $\gamma\neq I$.\\

(ii)$\overline{\Gamma_{C_3}(\mathcal{F}_{C_3})}=\mathbb{H}$.\\

For (i), first write $\gamma=T_1T_2\cdots T_m$, where each $T_i$ comes from one of the parabolic subgroups $B_{-1},B_{0} \text{ or } B_1$.  We say the \emph{length} of this word is $m$. We assume the length of the word is minimal so that $T_i,T_{i+1}$ are not in a same parabolic subgroup.  Then one can prove that $\gamma(\mathcal{F}_{C_3})$ lies in one of the regions from $\rm{I},\rm{II},\rm{III}$, which is determined by $T_1$.  Since $\rm{I},\rm{II},\rm{III}$ are disjoint from $\mathcal{F}_{C_3}$, (i) is thus proved.\\

For (ii), suppose $z\in\overline{\Gamma_{C_3}(\mathcal{F}_{C_3})}=\mathbb{H}$, we want to show that $z$ lies also in the interior of $\overline{\Gamma(\mathcal{F}_{C_3})}$.  First one can check that for each side of $\mathcal{F}_{C_3}$, there's one element $\gamma$ from $M_1,M_1^{-1},M_3,M_3^{-1},M_3M_5^{-1},M_3^{-1}M_5$ such that $\gamma(\mathcal{F}_{C_3})$ and $\mathcal{F}_{C_3}$ share this given side.  Now we place a ball of radius $\epsilon$ sitting at each of the cusps $-1,0,1$ and we say the complement of these balls to $\mathcal{F}_{C_3}$ the \emph{compact part} of $\mathcal{F}_{C_3}$, denoted by $\mathcal{F}_{C_3}^{c,\epsilon}$.  We define the compact part of $\gamma(\mathcal{F}_{C_3})$ simply by $\gamma(\mathcal{F}_{C_3}^{c,\epsilon})$.  Then there exists a universal constant $l(\epsilon)$ such that if $z$ lies within the $l(\epsilon)$ distance of some $\gamma(\mathcal{F}_{C_3}^{c,\epsilon})$, then $z$ lies either within some $\gamma^{'}(\mathcal{F}_{C_3})$ next to $\gamma(\mathcal{F}_{C_3})$, or on the common boundary of these two domains.  In both cases $z$ is an inner point of $\overline{\Gamma_{C_3}(\mathcal{F}_{C_3})}$.

It's an elementary geometric exercise to check that any $\gamma\in\Gamma_{C_3}$ will send these $\epsilon$-balls to balls with radii no greater than $\epsilon$ (by induction on the minimal length of word).  This means that if we choose $\epsilon=\frac{\text{Im} z}{10}$ and some $l(\epsilon)<\frac{\text{Im} z}{10}$, and some $\gamma_{\epsilon}$ such that $d(\gamma_{\epsilon}(\mathcal{F}_{C_3})),z)<\min\{\frac{\text{Im} z}{10},l(\epsilon)\}$, then $d(\gamma_{\epsilon}(\mathcal{F}_{C_3}^{c,\epsilon})),z)<\frac{\text{Im} z}{10}$.  In other words, $z$ is very close to the compact part of the fundamental domain $\gamma_{\epsilon}(\mathcal{F}_{C_3})$.  Therefore $z$ is an inner point.  Since $\overline{\Gamma(\mathcal{F}_{C_3})}$ is both open and closed, $\overline{\Gamma(\mathcal{F}_{C_3})}=\mathbb{H}$.
\end{proof}

 \begin{figure}[htbp] %  figure placement: here, top, bottom, or page
    \centering
    \includegraphics[width=3.5in]{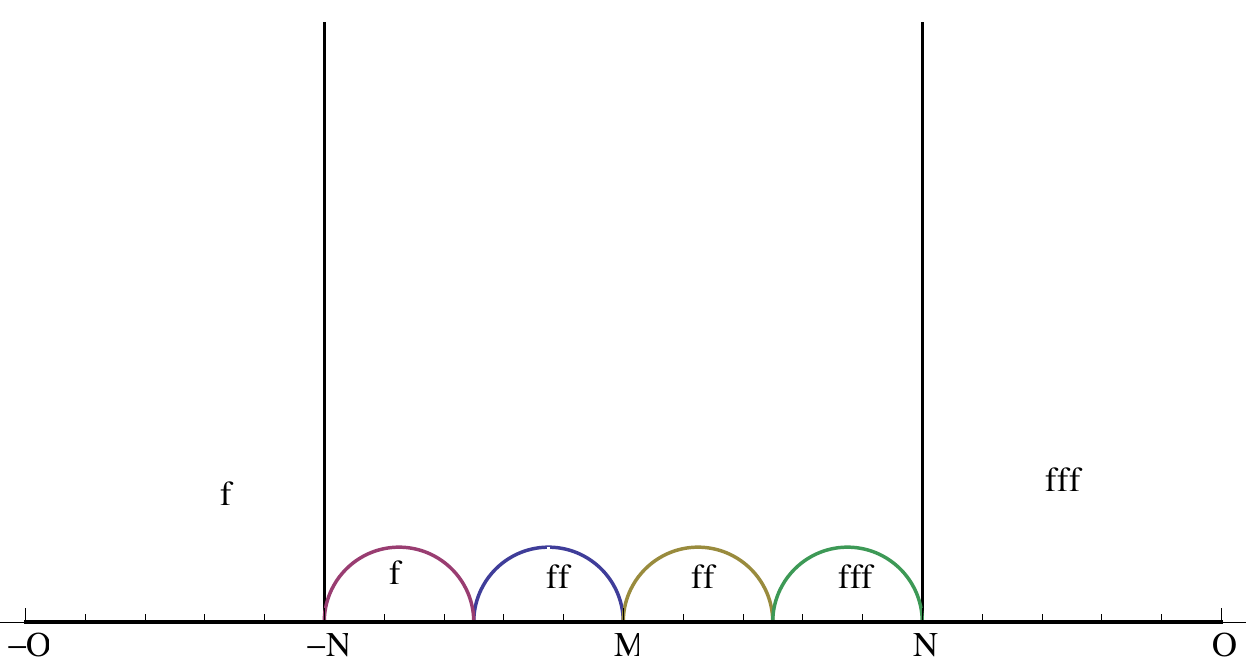} 
   \caption{The fundamental domain for $\Gamma_{C_3}$}
    \label{fundamentaldomaincongruence}
 \end{figure}

Conjugating $\Gamma_{C_3}$ by $\mat{1&0\\\sqrt{2}i&1}$, one gets $\Gamma_{C_1}=\mat{1&0\\\sqrt{2}i&1}\Gamma_{C_3}\mat{1&0\\-\sqrt{2}i&1}=\langle M_2,M_3,M_6\rangle$, which is a subgroup of $\Gamma$ fixing $C_{1}$. Similarly, $$\Gamma_{C_{3^{'}}}=\mat{-1&1+\sqrt{2}i\\-1&-1+\sqrt{2}i}\Gamma_{C_{3}}\mat{-1&1+\sqrt{2}i\\-1&-1+\sqrt{2}i}^{-1}=<M_7^{-1}M_3,M_7^{-1}M_5,M_7^{-1}M_6>,$$
 which is a subgroup fixing $C_{3^{'}}$.

 Let  \al{\label{akq}A_{k}(q)=\{g_1h_1j_1\hdots g_kh_kj_k:g_1,\hdots,g_k\in \Gamma_{C_3},h_1,\hdots,h_k\in\Gamma_{C_1},j_1,\hdots,j_k\in\Gamma_{C_{3^{'}}}\}} We have the following proposition: 
 
 \begin{prop}\label{gammaproduct} Let $q=\prod_{i}p_{i}^{n_{i}}$, then $\Gamma/\Gamma(q)=A_{10^{9}}(q).$
 \end{prop} 
Before proving Proposition \ref{gammaproduct}, we prove a few lemmas first.

\begin{lem}\label{pcase}
If $p\geq5$, then $A_{54}(p^m)=\Gamma/\Gamma(p^m)$.
\end{lem} 
\begin{proof}
Since $\Gamma_{C_1}$ is a congruence subgroup of level 4, we have $\Gamma_{C_1}/\Gamma_{C_1}(p^m)=SL(2,\ZZ/p^m\ZZ). $ We also have 
$$\mat{b^{-1}&0\\0&b}\cdot\mat{-\frac{1}{2}&0\\-\frac{7}{4}&-2}\cdot M_2^2\cdot\mat{1&0\\-\frac{1}{4}&1}\cdot M_2^{-1}\cdot\mat{b&0\\0&b^{-1}}=\mat{1&0\\3\sqrt{2}b^2i&1}.$$
Now we show that $\forall M>1$, we can find at most four elements $a,b,c,d\in\ZZ/p^m\ZZ)$ such that
$$a^2+b^2+c^2+d^2\equiv M(\text{mod }p^m)$$
This is true for $m=1$ by the Lagrange's Four Square Theorem, which states that every integer can be written as a sum of at most four squares of integers.  Choose $M^{'}\equiv M(p)$ with $0< M^{'}\leq p$, then we can choose $a^{'},b^{'},c^{'},d^{'}$ such that 
\al {\label{lagrange} {{a^{'}}^2+{b^{'}}^2+{c^{'}}^2+{d^{'}}^2=M^{'}},}
Necessarily all ${a^{'}},{b^{'}},{c^{'}},{d^{'}}$ have to be strictly less than $p$, and at least one of them is not zero, thus invertible in $\ZZ/p\ZZ$. So when mod $p$, $({a^{'}},{b^{'}},{c^{'}},{d^{'}})$ is a regular point on the curve \al{\label{pregular}x^2+y^2+z^2+w^2\equiv M^{'}(\text{mod }p).}   The general case follows from Hensel's lemma by lifting the solution $({a^{'}},{b^{'}},{c^{'}},{d^{'}})$ of \eqref{pregular}
to a solution $(a,b,c,d)$ of $$x^2+y^2+z^2+w^2=M(\text{mod }p^m)$$
This shows that 
$$\mat{1&0\\a\sqrt{2}i&1}\in A_{9}(p^m)$$
Multiplying the above matrix by $\mat{1&0\\b&1},b\in\ZZ/(p^m)$,  which can be found in $\Gamma_{C_1}$ since it contains $\mat{1&0\\4&1}$, we have
 $$\mat{1&0\\c&1}\in A_{9}(p^m)$$ for any $c\in\ZZ[\sqrt{2}i]/(p^m)$.  Conjugating the above element by 
$\mat{0&-1\\1&0}$, which is also congruent to some element in $\Gamma_{C_3}(\text{mod }p^m)$, we have
$$\mat{1&c\\0&1}\in A_{12}(p^m)$$
for any $c\in\ZZ[\sqrt{2}i]/(p^m)$. Now 
$$\mat{1&a\\0&1}\cdot\mat{1&0\\b&1}\cdot\mat{1&c\\0&1}=\mat{1+ab&a+c+abc\\b&1+bc}.$$
This shows that 
$$\mat{a^{'}&b^{'}\\c^{'}&d^{'}}\in\Gamma/\Gamma(p^m)$$ 
for any $c^{'}$ invertible in $\ZZ[\sqrt{2}i]/(p^m)$ and $a^{'}d^{'}-b^{'}c^{'}=1$.  There are $p^{3m-1}(p-1)$ such elements.  The size of $SL(2,\ZZ[\sqrt{2}i]/(p^m))$ is $p^{3m-3}(p-1)(p^2-1)$, which is strictly less than twice of $p^{3m-1}(p-1)$, this means that $A_{54}(p^m)$ has to be full of the group $SL(2,\ZZ[\sqrt{2}i]/(p^m))$.
\end{proof}

\begin{lem}\label{p23}
$A_{10^7}(2^m)=\Gamma/\Gamma(2^m)$, and $A_{10^7}(3^m)=\Gamma/\Gamma(3^m)$
\end{lem}
\begin{proof}
We prove the case when $p=2$ and explain the difference when $p=3$.
For $p=2$, we first prove the following claim by induction:\\

$\bd{Claim}$: For every $m\geq6$ and $g\in\Gamma(2^6)/\Gamma(2^m)$, we can find $g_1,g_2,g_3\in\Gamma_{C_3}(2^3)/\Gamma_{C_3}(2^m)$ such that 
$$g=g_1M_2g_2M_2^{-1}M_2^2g_3M_2^{-2}.$$

For $m=6$ we can choose $g_1=g_2=g_3=1$.  For $m>6$, we now assume this holds for $m-1$.  By the induction hypothesis, there exists $h_1,h_2,h_3\in\Gamma(2^3)$ such that
$$g=h_1M_2h_2M_2^{-1}M_2^2h_3M_2^{-2}+2^{m-1}x(\text{mod }2^m)$$
Now we choose some $x_i\in\text{Mat}(2,\ZZ)$ such that $x_i\equiv 0(2^{m-3})$ and $\text{tr}(x_i)\equiv 0(\text{mod }2^m)$ for $i=1,2,3$.  We have
\aln{g&\equiv(h_1+x_1)M_2(h_2+x_2)M_2^{-1}M_2^2(h_2+x_3)m_2^{-2}\\&-(x_1+M_2x_2M_2^{-1}+M_2^2x_2M_2^{-1})+2^{m-1}x(\text{mod }2^m)}
Since $\text{Det}(x_i+h_i)=1(2^m)$ and $x_i+h_i\equiv I (2^{3})$, $x_i+h_i$ is congruent to some element $g_i\in \Gamma_{C_3} (\text{mod } 2^m)$ using the congruence property of $\Gamma_{C_3}$.  The matrices $x_1,x_2,x_3$ can be chosen as a suitable linear combination of the matrices in the following calculations to cancel the term $2^{m-1}x$:
\al{
\nonumber& 2^{m-1}\mat{0&1\\0&0}+M_20M_2^{-1}+m_2^20M_2^{-2}\equiv 2^{m-1}\mat{0&1\\0&0}(\text{mod }2^m)
\\
\nonumber&2^{m-1}\mat{0&0\\1&0}+M_20M_2^{-1}+m_2^20M_2^{-2}\equiv 2^{m-1}\mat{0&0\\1&0}(\text{mod }2^m)
\\
\nonumber&2^{m-1}\mat{1&0\\0&-1}+M_20M_2^{-1}+m_2^20M_2^{-2}\equiv 2^{m-1}\mat{1&0\\0&-1}(\text{mod }2^m)
\\
\nonumber&2^{m-3}\mat{2&-1\\4&-2}+M_22^{m-3}\mat{0&0\\1&0}M_2^{-1}+M_2^20M_2^{-2}\equiv 2^{m-1}\mat{0&\sqrt{2}i\\0&0}(\text{mod }2^m)\\
\nonumber&2^{m-3}\mat{-2&4\\-1&2}+M_22^{m-3}\mat{0&0\\1&0}M_2^{-1}+M_2^20M_2^{-2}\equiv 2^{m-1}\mat{\sqrt{2}i&0\\\sqrt{2}i&-\sqrt{2}i}(\text{mod }2^m)\\
\nonumber&2^{m-3}\mat{4&0\\1&4}+M_20M_2^{-1}+M_2^22^{m-3}\mat{0&0\\1&0}M_2^{-2}\equiv 2^{m-1}\mat{0&0\\1&0}(\text{mod }2^m)
}

Thus we showed that 
$$A_3(2^m)\supseteq \Gamma(2^6)/\Gamma(2^m).$$
Now since the index of $\Gamma(2^6)/\Gamma(2^m)$ in $\Gamma/\Gamma(2^m)$ is $|\Gamma/\Gamma(2^6)|=2^{26}$, this implies that 
\al{A_{10^7}(2^m)=\Gamma/\Gamma(2^m)}
For the case $p=3$, the proof goes in the same way.  We choose the linear combinations of the following:
\aln
{\nonumber
&3^{m-1}\mat{0&1\\0&0}+M_20M_2^{-1}+M_2^20M_2^{-2}\equiv 3^{m-1}\mat{0&1\\0&0}(\text{mod }3^m)\\
&3^{m-1}\mat{0&0\\1&0}+M_20M_2^{-1}+M_2^20M_2^{-2}\equiv 3^{m-1}\mat{0&0\\1&0}(\text{mod }3^m)\\
&3^{m-1}\mat{1&0\\0&-1}+M_20M_2^{-1}+M_2^20M_2^{-2}\equiv 3^{m-1}\mat{1&0\\0&-1}(\text{mod }3^m)\\
&3^{m-1}\mat{0&1\\-1&0}+M_23^{m-1}\mat{0&1\\0&0}M_2^{-1}+M_2^20M_2^{-2}\equiv 3^{m-1}\mat{0&-\sqrt{2}i\\-\sqrt{2}i&0}(\text{mod }3^m)\\
&3^{m-1}\mat{0&1\\0&0}+M_23^{m-1}\mat{0&0\\1&0}M_2^{-1}+M_2^20M_2^{-2}\equiv 3^{m-1}\mat{\sqrt{2}i&0\\0&-\sqrt{2}i}(\text{mod }3^m)\\
&3^{m-1}\mat{-1&1\\1&1}+M_20M_2^{-1}+M_2^23^{m-1}\mat{0&0\\1&0}M_2^{-2}\equiv 3^{m-1}\mat{\sqrt{2}i&\\-\sqrt{2}i&-\sqrt{2}i}(\text{mod }3^m)\\
}
The constant $10^7$ also works in this case.
\end{proof}

Now we are able to prove Proposition \ref{gammaproduct}.
\begin{proof}[Proof of Proposition \ref{gammaproduct}]
First we embed $\Gamma/\Gamma(d)$ into $\prod_{p_i^{m_i}||d}\Gamma/\Gamma(p^m)$.  For any $x\in\prod_{p_i^{m_i}||}\Gamma/\Gamma(p^m)$, from Lemma \ref{pcase} and Lemma \ref{p23}, we can write
$$x\equiv\prod_{j=1}^{10^7}\gamma_{j,C_3}^{(i)}.\gamma_{j,C_1}^{(i)}.\gamma_{j,C_{3^{'}}}^{(i)}(p_i^{m_i})$$
for each $i$, where $\gamma_{j,C_3}^{(i)}\in\Gamma_{C_3},\gamma_{j,C_1}^{(i)}\in\Gamma_{C_1},\gamma_{j,C_{3^{'}}}^{(i)}\in\Gamma_{C_{3^{'}}}$.  Since $\Gamma_{C_3}$ is a congruence subgroup and $\Gamma_{C_1},\Gamma_{C_{3^{'}}}$ are conjugate to $\Gamma_{C_1}$, we can find $\gamma_1,\gamma_2,\gamma_3$ such that 
\aln
{&\gamma_1\equiv\gamma_{j,C_3}^{i}(\text{mod }p_i^{m_i})\\
&\gamma_2\equiv\gamma_{j,C_1}^{i}(\text{mod }p_i^{m_i})\\
&\gamma_3\equiv\gamma_{j,C_{3^{'}}}^{i}(\text{mod }p_i^{m_i})
}
for each $i$.  So $x=\gamma_1\gamma_2\gamma_3\in\Gamma/\Gamma(d)$.  So we have $$\prod_{p^m||d}\Gamma/\Gamma(p^m)=\Gamma/\Gamma(d)$$
\end{proof}

From the above proposition,  it follows directly that

\begin{lem}\text{ }\\
(1) If q=$\prod_{i}p_i^{m_i}$, then $\Gamma/\Gamma(q) \cong\prod_{i}\Gamma/\Gamma(p_i^{m_i})$,\\
(2) If $(q,6)=1$, then $\Gamma/\Gamma(q)=SL(2,(\ZZ[\sqrt{2}i]/(q)))$.\\
(3) If $l\geq3$, then the kernel of $\Gamma/\Gamma(2^l)\longrightarrow \Gamma/\Gamma(8)$ is the full of the kernel of $SL(2,\ZZ[\sqrt{2}i]/(2^l))\longrightarrow SL(2,\ZZ[\sqrt{2}i]/(8))$; If $l\geq 1$, then the kernel of $\Gamma/\Gamma(3^l)\longrightarrow\Gamma/\Gamma(3)$ is the full of the kernel of $SL(2,\ZZ[\sqrt{2}i]/(3^l))\longrightarrow SL(2,\ZZ[\sqrt{2}i]/(3))$.
\end{lem}

Since $\rho:SL(2,\ZZ[\sqrt{2}i]/(p_i^{m_i}))\longrightarrow SO_Q(\ZZ/p_i^{m_i}\ZZ)$ is surjective for each $i$,  the above theorem also holds for $\Aa$.  We state it here:

\begin{lem}\label{aproduct}\text{}\\
(1) If q=$\prod_{i}p_i^{n_i}$, then $\Aa/\Aa(q) \cong\prod_{i}\Aa/\Aa(p_i^{n_i})$,\\
(2) If (q,6)=1, then $\Aa/\Aa(q)=SO_Q(\ZZ/q\ZZ)$.\\
(3) If $l\geq3$, then the kernel of $\Aa/\Aa(2^l)\longrightarrow \Aa/\Aa(8)$ is the full of the kernel of $SO_Q(\ZZ/2^l\ZZ)\longrightarrow SO_Q(\ZZ/8\ZZ)$; If $l\geq 1$, then the kernel of $\Aa/\Aa(3^l)\longrightarrow\Aa/\Aa(3)$ is the full of the kernel of $SO_Q(\ZZ/3^l\ZZ)\longrightarrow SO_Q(\ZZ/3\ZZ)$.
\end{lem}

Now we can study the local obstruction of $\PC$.  We let $V$ be the set of vectors $\Gamma\cdot\bd{r}$ and $V_d$ be the reduction of $V\text{(mod }d)$.
We define $C_{p^m}$ as follows:
\begin{itemize}
\item if $p\geq 3$,
$$C_{p^m}=\{\bd{v}\in(\ZZ/p^m\ZZ)^4|Q(\bd{v})\equiv 0(\text{mod }p^m)\}$$
\item if p=2,
$$C_{2^m}=\{\bd{v}\in(\ZZ/2^m\ZZ)^4|Q(\bd{v})\equiv 0(\text{mod }2^m),\exists \bd{w}\equiv \bd{v}(2^{m}), Q(\bd{w})\equiv 0(\text{mod }2^{m+1})\}$$
\end{itemize}
Let $$\pi_{p^m}:C_{p^{m+1}}\longrightarrow C_{p^m}$$
be the canonical projection.  We have following lemmata: 
\begin{lem}\label{cp}If $p\geq 5$, then 
$$V_{p^m}=C_{p^m}$$
\end{lem}
\begin{proof}
This follows from Lemma \ref{aproduct},  and the fact that $SO_Q(\ZZ/p^m\ZZ)$ acts transitively on $C_{p^m}$.
\end{proof}

When $p=2,3$, the argument in Lemma \ref{cp} does not work because $\Gamma$ reduced at these local places is not the full group of $SL_2$.  But in each case the lifting will saturate for some finite $m$, as shown in Lemma \ref{c3} and \ref{c2}. In the case $p=3$, $\Gamma(\ZZ[\sqrt{2}i]/(3^m))$ is actually big enough to make $V_{3^m}=C_{3^m}$:
\begin{lem}\label{c3}
If $p=3$, then 
$$V_{3^m}=C_{3^m}.$$
\end{lem}
\begin{proof}
Using a program, we can check that $|V_3|=|C_3|=27$, and moreover, there exist $T_1,\hdots ,T_{27}\in \Aa\cap SO_{Q}(\ZZ)(3)$ such that all the solutions of $Q(\bd{v})\equiv 0(\text{mod }9)$ lying above $\bd{r}$ is given by:
\aln{
 T_1(\bd{r})=\bd{r}+&(T_1-I)\bd{r} \text{ (mod 3)}\\
&\vdots\\
 T_{27}(\bd{r})=\bd{r}+&(T_{27}-I)\bd{r} \text{ (mod 3)}.
}
Then for any $ m\geq 0$ the liftings from $V_{3^m}$ to $V_{3^{m+1}}$ are given by 
\aln{
 T_1^{3^{m}}(\bd{r})=\bd{r}+&(T_1-I)^{3^m}\bd{r} \text{ (mod }3^{m+1})\\
&\vdots\\
 T_{27}^{3^{m}}(\bd{r})=\bd{r}+&(T_{27}-I)^{3^m}\bd{r} \text{ (mod }3^{m+1})}
 We find that
$|V_{3^m}|=|C_{3^m}|$.
 \end{proof}
 
 \begin{lem} \label{c2}If $p=2$, then for $m\geq 3$,
 $$\pi_{2^{m+1}}^{-1}(V_{2^m})=V_{2^{m+1}}$$
 \end{lem}
\begin{proof}
We prove this by effective lifting.  This argument is due to Fuchs \cite{Fu10}.  For $n\geq 3 $, let $W(m)=(S_{1^{'}23}.S_{1^{'}2^{'}3})^{2^{m-3}}, X(m)=(S_{12^{'}3}.S_{12^{'}3^{'}})^{2^{m-3}},Y(m)=(S_{123^{'}}.S_{1^{'}2^{'}3})^{2^{m-4}}$. Then
\aln{
&W(n)=\mat
{1&0&0&2^{m-1}\\
2^{m-1}&1+2^{m-1}&2^{m-1}&2^{m-1}\\
0&0&1&0\\
0&2^{m-1}&0&1+2^{m-1}
},\\
&X(m)=\mat
{1&0&0&0\\
0&1&0&0\\
2^{m-1}&2^{m-1}&1+2^{m-1}&2^{m-1}\\
0&0&2^{m-1}&1+2^{m-1}
},\\
&Y(m)=\mat{1-2^{m-2}&-2^{m-2}&2^{m-2}&-2^{m-2}\\
-2^{m-2}&1-2^{m-2}&2^{m-2}&-2^{m-2}\\
2^{m-2}&-2^{m-2}&1+2^{m-2}&-2^{m-2}\\
-2^{m-2}&-2^{m-2}&2^{m-2}&1+2^{m-2}
}.
}

Then for example if $\bd{r}\equiv\langle3,2,2,3\rangle\text{(mod 4)}$, then 
\aln{
&I\bd{r}\equiv\bd{r}+2^{m-1}\langle0,0,0,0\rangle(\text{mod } 2^m)\\
&W(m)\bd{r}\equiv\bd{r}+2^{m-1}\langle1,0,0,1\rangle(\text{mod } 2^m)\\
&X(m)\bd{r}\equiv\bd{r}+2^{m-1}\langle0,0,0,1\rangle(\text{mod } 2^m)\\
&Y(m)\bd{r}\equiv\bd{r}+2^{m-1}\langle1,1,1,0\rangle(\text{mod } 2^m)
\\
&W(m)X(m)\bd{r}\equiv\bd{r}+2^{m-1}\langle1,0,0,0\rangle(\text{mod } 2^m)
\\
&W(m)Y(m)I\bd{r}\equiv\bd{r}+2^{m-1}\langle0,1,1,1\rangle(\text{mod } 2^m)
\\
&X(m)Y(m)\bd{r}\equiv\bd{r}+2^{m-1}\langle1,1,1,1\rangle(\text{mod } 2^m)
\\
&W(m)X(m)Y(m)\bd{r}\equiv\bd{r}+2^{m-1}\langle0,1,1,0\rangle(\text{mod } 2^m)
}
\end{proof}

Collecting the result from Lemma \ref{cp} to Lemma \ref{c2},  we obtain the following proposition which describes the local structure of $V$.

\begin{thm}\label{localobstruction}\text{ }\\
(1)$V_q\cong \prod_{i}V_{p_i^{n_i}}$,\\
(2)$\pi_{p^{m+1}}^{-1}(V_{p^{m}})=V_{p^{m+1}}$ for $p\geq 3$ and $m\geq 0$,\\
(3)$\pi_{2^{m+1}}^{-1}(V_{2^{m}})=V_{2^{m+1}}$ for $p=2$ and $m\geq 3$.
\end{thm}

Lemma \ref{local}, thus Theorem \ref{reduction} then follow directly from Theorem \ref{localobstruction} because the first three components of $V$ are curvatures. \\

Now we prove Theorem \ref{spec}.  Bourgain, Gamburd and Sarnak \cite{BGS11} established an equivalence between a geometric spectral gap and  a combinatorial spectral gap for a finitely generated Fuchsian group $F$. Let $S$ be a finite symmetric ($S=S^{-1}$) generating set of $F$.  For each $q$,  we have a Cayley graph of $F/F(q)$ over S.  There's a Markov operator (which is a discrete version of Laplacian) on the functions of this Cayley graph.  A Combinatorial spectral gap is then a uniform positive lower bound of the distance between the biggest two eigenvalues $\lambda_0^{'}(q)=1$ and $\lambda_1^{'}(F(q),S)$ of this operator.  Later this equivalence is generalized by Kim \cite{Ki11} to Kleinian groups,  which applies to our case $\Gamma$.   From the celebrated Selberg's $\frac{3}{16}$ theorem we know there are geometric spectral gaps for $\Gamma_{C_3},\Gamma_{C_1},\Gamma_{C_{3^{'}}}$.  It then follows that the combinatorial gaps exist for these groups from \cite{BGS10}.  Now we apply Varj\"{u}'s lemma in the Appendix of \cite{BK13}:
\begin{lem}[Varj\"{u}]\label{varju}
Let $G$ be a finite group and $S\subset G$ a finite symmetric generating set.  Let $G_1,G_2,...,G_k$ be subgroups of $G$ such that for every $g\in G$ there are $g_1\in G_1,\hdots,g_k\in G_k$ such that $g=g_1\hdots g_k$.  Then
$$1-\lambda_1^{'}(G,S)\geq \min_{1\leq i\leq k}\left\{\frac{|S\cap G_i|}{|S|}.\frac{1-\lambda_1^{'}(G_i,S\cap G_i)}{2k^2}\right\}$$
\end{lem}  

In our case $G$ is $\Gamma$(mod $q$),  $G_i$'s are $\Gamma_{C_3},\Gamma_{C_1}$ or $\Gamma_{C_{3^{'}}}$(mod $q$), in light of Proposition \ref{gammaproduct}.  And we let $S$ to be the union of $M_1,M_2,M_3,M_4,M_5,M_6,M_7^{-1}M_3$ and their inverses.  Clearly Lemma \ref{varju} provides a spectral gap for $\Gamma$. This implies a geometric spectral gap for $\Gamma$ again by \cite{BGS11}.

\section{Circle Method} % circle method chapter

In this chapter we are proving Theorem \ref{mainthm} via the Hardy-Littlewood circle method.  In \S 3.1 we set up the ensemble for the circle method.  In \S 3.2 we do major arc analysis, where we crucially use the spectral gap property of $\Gamma$ for several counts.  From \S3.3 to \S3.5 we do minor arc analysis, and several Kloosterman-type sums naturally appear here. \S 3.6 gathers all the previous results and finishes the proof of Theorem \ref{mainthm}.
\subsection{Setup of the circle method}
Recall that $\Gamma_{C_3}$ is a congruence subgroup $$\left\{\mat{a&b\\c&d}\in SL(2,\ZZ)|a\equiv d\equiv 1(\text{mod }2),b\equiv c\equiv0\text{ or }2(\text{mod }4)\right\}.$$  Therefore,  for any $x,y\in\ZZ$ with $(x,2y)=1$, we can find an element $\xi_{x,y}$ of the form $\mat{x&2y\\**&*}\in\Gamma_{C_3}$.  Under the spin homomorphism $\rho$,  $\xi_{x,y}$ will be mapped to 
$$\mat{x^2-y^2&-1+x^2+y^2&2xy&-2xy+2y^2\\0&1&0&0\\**&*&*&*\\**&*&*&*}.$$
Hence we have the following theorem:
\begin{thm}\label{shiftedquadraticform}
Let $x,y\in\ZZ$ with $(x,2y)=1$, and take any element $\gamma\in{\Aa}$ with the corresponding quadruple
$$\bd{v_{\gamma}}=\gamma(\bd{r})=\langle a_{\gamma},b_{\gamma},c_{\gamma},d_{\gamma}\rangle.$$
Then the number
\al{\label{shiftedqform}\langle\bd{e_1},\xi_{x,y}.\gamma\bd{r}\rangle=A_{\gamma}x^2+2B_{\gamma}xy+C_{\gamma}y^2-b_{\gamma}}
is the curvature of some circle in $\PC$, where
\al{
\nonumber&A_{\gamma}:=a_{\gamma}+b_{\gamma}\\
\nonumber&B_{\gamma}:=c_{\gamma}-d_{\gamma}\\
&C_{\gamma}:=-a_{\gamma}+b_{\gamma}+2d_{\gamma}
}
\end{thm}
We can view \eqref{shiftedqform} as a shifted quadratic form $\ff(x,y)$ determined by $\gamma$ with variables $x,y$.  We define 
\aln{&\ff(x,2y)=\langle\bd{e}_1,\xi_{x,y}.\gamma(\bd{r})\rangle\\
&\fff(x,2y)=A_{\gamma}x^2+2B_{\gamma}xy+C_{\gamma}y^2}
Then $\ff=\tilde{\ff}-b_{\gamma}$, and the discriminant of $\fff$ is $-8b_{\gamma}^2$.  \\

Now we set up our ensemble for the circle method.   Let  $N$ be the main growing parameter.  Write $N=TX^2$,  where $T=N^{\frac{1}{200}}$, a small power of $N$, and $X=N^{\frac{199}{400}}$.  We define our ensemble to be a subset of ${\Aa}$ (with multiplicity) of Frobenius norm $\asymp N$.  The ensemble is a product of a subset $\mathfrak{F}$ of norm $T$, and a subset $\mathcal{X}$ of norm $X^2$. We further write $T=T_1T_2$, where $T_2=T_1^{\mathcal{C}}$ and $\mathcal{C}$ is a large number which is determined in Lemma \ref{bk2} .   We define $\FF$ in the following way:

\aln
{\FF=\FF_{T}=\left\{\gamma=\gamma_1\gamma_2:\begin{array}{ccc} &\gamma_1,\gamma_2\in{\Aa}\\&T_1<||\gamma_1||<2T_1\\&T_1<||\gamma_2||<2T_2\\&<\bd{e}_2,\gamma_1\gamma_2\bd{r}>\frac{T}{100}\end{array}\right\}
}

Recall that the Hausdorff dimension of the circle packing $\delta$ is strictly greater than 1.  The size of $\FF$ is $\asymp T^{\delta}$, which can be seen from \cite{KO11}.  The last condition in the definition of $\FF$ implies that $b_\alpha\asymp T$,  which is crucial in our minor arc analysis later.  The subset of norm $X^2$ is the image of some elements of the form
$\mat{x&2y\\**&*}$ in $\Gamma_{C_3}$, with $x,y\asymp X$, under the map $\rho$.

For technical reasons we need to smooth the variables $x$ and $y$.  We fix a smooth, nonnegative function $\psi$ which is supported in $[1,2]$ and $\int_{\RR}\psi(x)\text{d}x=1$.  Our main goal is to study the following representation number

\al{\rn:=\sum_{\ff\in\FF_{T}}\sum_{\sk{x,y\in\ZZ(x,2y)=1}}\psi\left(\frac{x}{X}\right)\psi\left(\frac{2y}{X}\right)\bd{1}_{\{n=\ff(x,2y)\}}
}

via its  Fourier transform:
\al{
\hrn(\theta):=\sum_{\ff\in\FF_T}\sum_{\sk{x,y\in\ZZ(x,2y)=1}}\psi\left(\frac{x}{X}\right)\psi\left(\frac{2y}{X}\right)e(\theta\ff(x,2y))
}
$\mathcal{R}_N$ and $\hrn$ is related by
\aln{
\rn=\int_0^1\hrn(\theta)e(-n\theta)d\theta.
}

Therefore,  $\rn\neq0$ implies $n$ is represented.  Since $\delta>1$, one expects roughly  that each admissible $n$ is represented by $T^{\delta-1}$ times.  One important thing for circle method here is that $T^{\delta-1}$ is a positive power of $N$, so we have enough solutions to play with.

Another technicality is that we replace the condition $(x,2y)=1$ by the M\"obius orthogonal relation:
\begin{equation*}
\sum_{d|n}\mu(n) =
\begin{cases}
1 & \text{if } n=1,\\
0 & \text{if } n>1.
\end{cases}
\end{equation*}
We introduce another parameter $U$ which is a small power of $N$.  It is determined in \eqref{determine}.  We then define the corresponding representation function 
$$\rnu(n):=\sum_{\ff\in\FF_T}\sum_{x,y\in\ZZ}\sum_{\substack{u|(x,2y)\\u<U}}\mu(u)\psi\left(\frac{x}{X}\right)\psi\left(\frac{2y}{X}\right)\bd{1}_{\{n=\ff(x,2y)\}}$$
and its Fourier transform:
$$\hrnu(\theta):=\sum_{\ff\in\FF_T}\sum_{x,y\in\ZZ}\sum_{\substack{u|(x,2y)\\u<U}}\mu(u)\psi\left(\frac{x}{X}\right)\psi\left(\frac{2y}{X}\right)e(\theta\ff(x,2y))$$
The $\ell^1$ norm of  $\mathcal{R}_N$ is $\asymp T^{\delta}X^2$.  We first show that the difference between $\mathcal{R}_N$ and $\rnu$ is small in $\ell^1$, compared to $T^{\delta}X^2$:
\begin{lem}\label{sdm}$$\sum_{n<N}\left|\mathcal{R}_N(n)-\rnu(n)\right|\ll_{\epsilon}  \frac{T^{\delta} X^{2+\epsilon}}{U}.$$
\end{lem}
\begin{proof}
\aln{
\sum_{n<N}|\mathcal{R}_N(n)-\rnu(n)|=&\sum_{n<N}\left|\sum_{\ff\in\FF_{T}}\sum_{(x,2y)=1}\sum_{\substack{u|(x,2y)\\u\geq U}}\mu(u)\psi(\frac{x}{X})\psi(\frac{2y}{X})\bd{1}_{\{n=\ff(x,2y)\}}\right|\\
\leq&\sum_{\ff\in\FF}\sum_{x,y\in\ZZ}\psi(\frac{x}{X})\psi(\frac{2y}{X})\left|\sum_{\substack{u|(x,2y)\\u\geq U}}\mu(u)\right|\\
\ll&\sum_{\ff\in\FF}\sum_{x\ll X}\sum_{\substack{u\geq U\\u|x}}\sum_{\substack{y\ll X\\2y\equiv 0(u)}}1
\ll \frac{T^{\delta} X^{2+\epsilon}}{U} 
}
\end{proof}

Now we decompose [0,1] into ``major" and ``minor' arcs according to the standard Diophantine approximation of real numbers by rationals.  Let $M=TX$ be the parameter controlling the depth of the approximation.  Write $\alpha=\frac{a}{q}+\beta$.  We introduce two parameters $Q_0,K_0$ such that the major arcs corresponds to $q\leq Q_0,\beta\leq\frac{K_0}{N}$.  Both $Q_0$ and $K_0$ are small powers of $N$, and they are determined in \eqref{determine}.\\

Next we introduce the ``hat'' function
\aln{\tf:=\text{min}(1+x,1-x)^{+}}
whose Fourier transform is
\aln{\hat{t}(y)=\left(\frac{\text{sin}(\pi y)}{\pi y}\right)^2.}
From $\tf$, we construct a spike function $\mathfrak{T}$ which captures the major arcs:
\aln{\mathfrak{T}(\theta):=\sum_{q\leq Q_0}\sum_{(r,q)=1}\sum_{m\in\ZZ}\tf\left(\frac{N}{K_0}\left(\theta+m-\frac{a}{q}\right)\right).}
The ``main" term is then defined to be:
\al{\label{mainterm}\mathcal{M}_N(n):=\int_0^1\mathfrak{T}(\theta)\hrn(\theta)e(-n\theta)d\theta}
and the ``error" term
\al{\enn:=\int_0^1(1-\Tf(\theta))\hrn(\theta)e(-n\theta)d\theta.}
We define $\mathcal{M}_N^U(n)$ and $\mathcal{E}_N^U(n)$ in a similar way. \\

Now we explain the general strategy to prove the Theorem {\ref{mainthm}}.
\al{\label{dia}
\begin{array}[c]{ccccc}
\mathcal{R}_N&=&\mathcal{M}_N&+&\mathcal{E}_N\\ 
\vert\scriptstyle{}&&\vert\scriptstyle{}&&\vert\\
\mathcal{R}_N^{U}&=&\mathcal{M}_N^{U}&+&\mathcal{E}_N^{U}
\end{array}
}

$\bd{STRATEGY}:$
\begin{enumerate}
\item The difference between $\mathcal{R}_N$ and $\rnu$ is small in $\ell^1$.  We have shown this in Lemma \ref{sdm}.
\item $\mathcal{M}_N$ is large for each $n$ admissible in the range $(\frac{N}{2}, N)$ (See Theorem \ref{majorarcanal}), and the difference of $\mathcal{M}_N$ and $\mathcal{M}_N^U$ is small in $\ell^2$ (See Lemma \ref{smm}).  This will be done in \S3.2.
\item  Step 2 will imply that the difference between $\mathcal{E}_N^{U}$ and $\mathcal{E}_{N}$ is also small.
\S 3.3 to \S 3.5 will show that the $\mathcal{E}_N^{U}$ is small in $l^2$ (See Theorem \ref{minorthm}), which implies that $\mathcal{E}_N$ is small in $\ell^1$.  This would greatly restrain the size of the set of admissible $n$'s where $\mathcal{R}_N(n)=0$ in $(\frac{N}{2}, N)$, because each term would contribute large to $\mathcal{E}_N$.
\end{enumerate}

\subsection{Major Arc Analysis} %major arc analysis 

From \eqref{mainterm},
\al{\label{mnnt}\nonumber\mnn&=\int_0^1\sum_{q<Q_0}\sum_{r(q)}{}^{'}\sum_{m\in\ZZ}\frak{t}\left(\frac{N}{K_0}\left(\theta+m-\frac{r}{q}\right)\right)\hrn(\theta)e(-n\theta)d\theta\\
\nonumber&=\int_{-\infty}^{\infty}\sum_{q<Q_0}\sum_{r(q)}{}^{'}\frak{t}\left(\frac{N}{K_0}\beta\right)\hrn\left(\beta+\frac{r}{q}\right)e\left(-n\left(\beta+\frac{r}{q}\right)\right)d\beta\\
&=\sum_{\sk{x,y\\(x,2y)=1}}\psi\left(\frac{x}{X}\right)\psi\left(\frac{2y}{X}\right)\sum_{q<Q_0}\sum_{r(q)}{}^{'}\sum_{\ff\in\FF}e\left(\frac{r}{q}(\ff(x,2y)-n)\right)\int_{-\infty}^{\infty}\tf\left(\frac{N}{K_0}\beta\right)e(\beta(\ff(x,2y)-n)d\beta
}
Now we cite Lemma 5.3 from \cite{BK13} to deal with the $\FF$ sum in \eqref{mnnt}.
\begin{lem}[Bourgain, Kontorovich]\label{bk1}Let $1<K<T_2^{\frac{1}{10}}$, fix $|\beta|<\frac{K}{N}$, and fix $x,y\asymp X$. Then for any $\gamma_0\in\Gamma$, any $q\geq1$, we have
$$\sum_{\gamma\in\FF\cap\gamma_0\Gamma(q)}e(\beta\ff_{\gamma}(x,2y))=\frac{1}{\Gamma:\Gamma(q)}\sum_{\ff\in\FF}e(\beta\ff_{\gamma}(x,2y))+O(T^{\Theta}K),$$
where $\Theta<\delta$ depends only on the spectral gap for $\Gamma$, and the implied constant does not depend on $q,\gamma_0,x$ or $y$.
\end{lem}

Returning to \eqref{mnnt}, we can decompose the set $\FF$ as cosets of $\Gamma(q)$.  Applying Lemma \ref{bk1} and setting $K=K_0$, we have
\al{\nonumber\mnn&=\sum_{\sk{x,y\in\ZZ\\(x,2y)=1}}\psi\left(\frac{x}{X}\right)\psi\left(\frac{2y}{X}\right)\sum_{q<Q_0}\sum_{r(q)}{}^{'}\sum_{\bar{\gamma}\in\Gamma/\Gamma(q)}e\left(\frac{r}{q}(\ff_{\bar{\gamma}}(x,2y)-n)\right)\\
\nonumber&\times\sum_{\sk{\gamma\in\FF\\\gamma\equiv\bar{\gamma}}}\int_{-\infty}^{\infty}\frak{t}\left(\frac{N}{K_0}\beta\right)e(\beta(\ff_{\gamma}(x,2y)-n))d\beta\\
\nonumber&=\sum_{\sk{x,y\in\ZZ\\(x,2y)=1}}\psi\left(\frac{x}{X}\right)\psi\left(\frac{2y}{X}\right)\sum_{q<Q_0}\sum_{r(q)}{}^{'}\sum_{\bar{\gamma}\in\Gamma/\Gamma(q)}e\left(\frac{r}{q}(\ff_{\bar{\gamma}}(x,2y)-n)\right)\\
\nonumber&\times\left(\frac{1}{[\Gamma:\Gamma(q)]}\sum_{\sk{\gamma\in\FF}}\int_{-\infty}^{\infty}\frak{t}\left(\frac{N}{K_0}\beta\right)e(\beta(\ff_{\gamma}(x,2y)-n))d\beta+O\left(\frac{T^{\Theta}K_0^2}{N}\right)\right)\\
\nonumber&=\sum_{r(q)}{}^{'}\psi\left(\frac{x}{X}\right)\psi\left(\frac{2y}{X}\right)\mathfrak{S}_{Q_0}(n)\frak{M}(n)+O\left(\frac{T^{\Theta}X^2K_0^2Q_0^8}{N}\right)\\
&=\sum_{r(q)}{}^{'}\psi\left(\frac{x}{X}\right)\psi\left(\frac{2y}{X}\right)\mathfrak{S}_{Q_0}(n)\frak{M}(n)+O\left(N^{-\eta}\right)
}
where $\eta_1>0$, as can be seen from \eqref{determine}, and
\al{\label{singularq0}
\frak{S}_{Q_0}(n)=\frak{S}_{Q_0;x,y}(n):&=\sum_{q<Q_0}\sum_{(r,q)=1}\frac{1}{[\Gamma:\Gamma(q)]}\sum_{\bar{\gamma}\in\Gamma/\Gamma(q)}e\left(\frac{r}{q}(\ff_{\bar{\gamma}}(x,2y))-n\right)\\
&=\sum_{q<Q_0}\frac{1}{[\Gamma/\Gamma(q)]}\sum_{\bar{\gamma}\in\Gamma/\Gamma(q)}c_q(\ff_{\bar{\gamma}}(x,2y)-n)
} and 
\al{\nonumber
\frak{M}(n):=\frak{M}_{x,y}(n)&:=\sum_{\sk{\gamma\in\FF}}\int_{-\infty}^{\infty}\frak{t}\left(\frac{N}{K_0}\beta\right)e(\beta(\ff_{\gamma}(x,2y)-n))d\beta\\
&=\frac{K_0}{N}\sum_{\sk{\gamma\in\FF}}\hat{t}\left(\frac{K_0}{N}(\ff(x,2y)-n)\right)
}\\

The function $c_q(n)$ in \eqref{singularq0} is the classical Ramanujan's sum, defined by
$$c_q(n)=\sum_{a(q)}{}^{'}e\left(\frac{an}{q}\right),$$ 
$c_q(n)$ is multiplicative with respect to $q$, and
\aln{c_{p^k}(n)=\begin{cases}0&\text{if } p^m||n,m\leq k-2,\\
-p^{k-1}& \text{if } p^{k-1}||n,\\
p^{k-1}(p-1)& \text{if } p^k|n.
\end{cases}}

Now $\frak{M}(n)\gg \frac{T^{\delta}}{N}$ for $\frac{N}{2}<n<N$, which can be seen from the following lemma by Lemma 5.4 in \cite{BK13}.  We record it here:
\begin{lem}[Bourgain, Kontorovich]\label{bkcao} Fix $N/2<n<N,1<K\leq T_2^{\frac{1}{10}}$, and $x,y\asymp X$.  Then 
$$\sum_{\gamma\in\FF}\bd{1}_{\{|\ff_{\gamma}(x,2y)-n|<\frac{N}{K}\}}\gg\frac{T^{\delta}}{K}+T^{\Theta},$$
where $\Theta<\delta$ depends only on the spectral gap for $\Gamma$.  The implied constant is independent of $x,y$ and $n$.
\end{lem}

Now we are at a position to analyze the non-Archimedean part $\frak{S}_{Q_0}$.  We push $\frak{S}_{Q_0}(n)$ to infinity, and define
\aln{\frak{S}(n):=\sum_{q=1}^{\infty}\frac{1}{[\Gamma:\Gamma(q)]}\sum_{\bar{\gamma}\in\Gamma/\Gamma(q)}c_q(\ff_{\bar{\gamma}}(x,2y)-n)\\
=\sum_{q=1}^{\infty}\sum_{a\in\ZZ/q\ZZ}\tau_q(a)c_q(a-n):=\sum_{q=1}^{\infty}B_q(n),}

where \aln{
\tau_q(a)=\frac{\#\{\langle u,v,w\rangle(\text{mod }q)|\langle a,u,v,w\rangle\in\mathcal{P}\}}{\#\{\langle x,u,v,w\rangle(\text{mod }q)|\langle x,u,v,w\rangle\in\mathcal{P}\}}
}

From Theorem \ref{localobstruction} we know that $\tau_q(n)$ is multiplicative in the $q$ variable, and so is $B_q(n)$.  Therefore,  we can formaly write 
$$\frak{S}(n)=\prod_{p}(1+B_p(n)+B_{p^2}(n)+\hdots)$$ 

For $p\geq 3$, by Theorem \ref{localobstruction},  we can show that
\aln{B_{p}(n)=\begin{cases}\frac{-1-p(\frac{-2}{p})}{p^2+(1+(\frac{-2}{p}))p+1}&\text{if}p|n,\\
\frac{p(\frac{-2}{p})+1}{p^3+p(p-1)(\frac{-2}{p})-1}&\text{if}p\nmid n,\end{cases}}
and $B_{p^k}=0$ for $k\geq 2$.  For $p=2$, we have $B_{2^m}=0$ for $m\geq 4$ and
\aln{1+B_2(n)+B_4(n)+B_8(n)=\begin{cases}8&\text{if }n\equiv\kappa_1(\text{mod }8)\\0&\text{otherwise.}\end{cases}}

Thus we see that $\frak{S}_{Q_0}$ is a non-negative function which is non-zero if and only if $n\equiv \kappa_1(\text{mod }8)$, which matches exactly the local obstruction described in Theorem \ref{localobstruction}.  For such admissible $n$'s,  $\frak{S}_{Q_0}$ satisfies $N^{-\epsilon}\ll_{\epsilon}\frak{S}_{Q_0}(n)\ll_{\epsilon}N^{\epsilon}$.\\

To analyze $\rnu$ we need to extend the definition of $\frak{S}_{x,y}(n)$ restricted to $(x,2y)=1$ to all pair of integers $x,y$. If $(x,2y)=u>1$, the same calculation shows that $\frak{S}_{Q_0;x,y}(n)$ has the same local factor for $p\neq 2$, and $B_{2^m}=0$ for $m\geq 4$.  Therefore,  $\frak{S}_{x,y}(n)\ll_{\epsilon}N^{\epsilon}$ for any $x,y\in\ZZ$.\\

The difference between $\frak{S}$ and $\frak{S}_{Q_0}$ is small.  In fact, we have
\al{\label{dss}|\frak{S}(n)-\frak{S}_{Q_0;x,y}(n)|\leq\sum_{q\geq Q_0}|B_q(n)|\leq\sum_{q_1:n}|B_{q_1}(n)|    \sum_{\sk{(q_2,q_1)=1\\q_1q_2\geq Q_0}}|B_{q_2}(n)|
}
Here we write $q=q_1q_2$, where $q_1:n$ means that $q_1$ is the product of all primes dividing $n$.   We also know that $B_q(n)$ as a function of $q$ is supported on (almost) square-free numbers (as can be see by the previous paragraphs), we have 
\aln{\eqref{dss}\ll\sum_{q_1:n}\frac{1}{q_1}\frac{q_1}{Q_0}\ll \frac{2^{w(n)}}{Q_0} 
}
where $w(n)$ denotes the number of primes dividing $n$.   Therefore, we conclude that if $n$ is admissible, then $N^{-\epsilon}\ll \mathfrak{S}_{Q_0}(n)\ll_{\epsilon}N^{\epsilon}$.   In summary, we have

\begin{thm} \label{majorarcanal}For $\frac{N}{2}<n<N$, there exists a function $\frak{S}_{Q_0}(n)$ such that if $n$ is admissible, then
\aln{\mnn \gg \frak{S}_{Q_0}(n)T^{\delta-1},}

where \aln{N^{-\epsilon}\ll_{\epsilon}\frak{S}_{Q_0}(n)\ll_{\epsilon}N^{\epsilon}.}

\end{thm}
 
Next we show that the difference of $\mn$ and $\mnu$ is small in $\ell^1$.
\begin{lem}\label{smm}
$$\sum_{\frac{N}{2}<n<N}|\mn(n)-\mnu(n)|\ll_{\epsilon}\frac{N^{\epsilon}X^2T^{\delta}}{U}+\frac{T^{\Theta}X^2K_0^2Q_0^2}{U},$$
where $\Theta$ is the same as in Lemma \ref{bk1}.
\end{lem}
\begin{proof}
Going in the same way as \eqref{mnnt} to unfold $\mnu(n)$, we have
\aln{\label{mdiv}
\nonumber\mn(n)-\mnu(n)&=\sum_{\sk{u\geq U\\u \text{ odd}}}\mu(u)\sum_{x,y\in\ZZ}\psi\left(\frac{xu}{X}\right)\psi\left(\frac{2yu}{X}\right)\sum_{q<Q_0}\sum_{r(q)}{}^{'}\sum_{\bar{\gamma}\in\Gamma/\Gamma(q)}e\left(\frac{r}{q}(\ff_{\bar{\gamma}}(xu,2yu)-n)\right)\\
\nonumber&\times\sum_{\sk{\gamma\in\FF\\\gamma\equiv\bar{\gamma}(\text{mod }\Gamma(q))}}\int_{-\infty}^{\infty}\tf\left(\frac{N}{K_0}\theta\right)e(\theta(\ff_{\gamma}(xu,2yu)-n))d\theta\\
\nonumber&+\sum_{\sk{u\geq U\\u \text{ even}}}\mu(u)\sum_{x,y\in\ZZ}\psi\left(\frac{xu}{X}\right)\psi\left(\frac{yu}{X}\right)\sum_{q<Q_0}\sum_{r(q)}{}^{'}\sum_{\bar{\gamma}\in\Gamma/\Gamma(q)}e\left(\frac{r}{q}(\ff_{\bar{\gamma}}(xu,yu)-n)\right)\\
\nonumber&\times\sum_{\sk{\gamma\in\FF\\\gamma\equiv\bar{\gamma}(\text{mod }\Gamma(q))}}\int_{-\infty}^{\infty}\tf\left(\frac{N}{K_0}\theta\right)e(\theta(\ff_{\gamma}(xu,yu)-n))d\theta\\
\nonumber&=\sum_{\sk{u\geq U\\u \text{ odd}}}\mu(u)\sum_{x,y\in\ZZ}\psi\left(\frac{xu}{X}\right)\psi\left(\frac{2yu}{X}\right)\frak{S}_{Q_0,(xu,2yu)}(n)\times \sum_{\gamma\in\FF }\frac{K_0}{N}\hat{\tf}\left(\frac{K_0}{N}(\ff_{\gamma}(xu,2yu)-n)\right)\\
\nonumber&+\sum_{\sk{u\geq U\\u \text{ even}}}\mu(u)\sum_{x,y\in\ZZ}\psi\left(\frac{xu}{X}\right)\psi\left(\frac{yu}{X}\right)\frak{S}_{Q_0,(xu,yu)}(n)\times \sum_{\gamma\in\FF }\frac{K_0}{N}\hat{\tf}\left(\frac{K_0}{N}(\ff_{\gamma}(xu,yu)-n)\right)\\&+O\left(\frac{T^{\Theta}X^2K_0^2Q_0^2}{NU}\right)
}
Therefore,
\aln{&\sum_{\frac{N}{2}<n<N}|\mn(n)-\mnu(n)|\\
&\ll \sum_{\sk{u>U\\u\text{ odd}}}\sum_{x,y\in\ZZ}\psi\left(\frac{xu}{X}\right)\psi\left(\frac{2yu}{X}\right)\frac{K_0}{N}\sum_{\ff\in\FF}\sum_{\frac{N}{2}<n<N}\frak{S}_{Q_0}(n)\hat{\tf}\left(\frac{K_0}{N}(\ff(xu,2yu)-n)\right)\\
&+ \sum_{\sk{u>U\\u\text{ even}}}\sum_{x,y\in\ZZ}\psi\left(\frac{xu}{X}\right)\psi\left(\frac{yu}{X}\right)\frac{K_0}{N}\sum_{\ff\in\FF}\sum_{\frac{N}{2}<n<N}\frak{S}_{Q_0}(n)\hat{\tf}\left(\frac{K_0}{N}(\ff(xu,yu)-n)\right)+\frac{T^{\Theta}X^2K_0^2Q_0^2}{U}\\
&\ll_{\epsilon}\frac{X^2T^{\delta}N^{\epsilon}}{U}+\frac{T^{\Theta}X^2K_0^2Q_0^2}{U}
}
\end{proof}

In light of \eqref{determine}, we have 
\aln{
\sum_{\frac{N}{2}<n<N}|\mn(n)-\mnu(n)|\llet T^{\delta}X^2N^{-\eta}
}

\subsection{Minor Arc Analysis I}% minor arc analysis 1

The rest few sections of the paper is dedicated to proving Theorem \ref{minorthm}, which shows that $(1-\Tf(\theta))\hrnu$ is small in $L^2$.  By Plancherel formula this will imply that $\enu$ is small in $\ell^2$, fulfilling Step 3 of our strategy.
\begin{thm} \label{minorthm}$$\int_{0}^{1}\left|(1-\Tf(\theta))\rnu(\theta)\right|^2d\theta \ll NT^{2{(\delta-1)}}N^{-\eta} $$
\end{thm}
We divide the integral into three parts.  \al{
&\mathcal{I}_1=\sum_{q<Q_0}\sum_{r(q)}{}^{'}\int_{\frac{r}{q}-\frac{1}{qM}}^{\frac{r}{q}+\frac{1}{qM}}|(1-\Tf(\theta))\hrnu(\theta)|^2d\theta\\
&\mathcal{I}_2=\sum_{Q_0\leq q<X}\sum_{r(q)}{}^{'}\int_{\frac{r}{q}-\frac{1}{qM}}^{\frac{r}{q}+\frac{1}{qM}}|(1-\Tf(\theta))\hrnu(\theta)|^2d\theta\\
&\label{i3}\mathcal{I}_3=\sum_{X\leq Q\leq M}\sum_{r(q)}{}^{'}\int_{\frac{r}{q}-\frac{1}{qM}}^{\frac{r}{q}+\frac{1}{qM}}|(1-\Tf(\theta))\hrnu(\theta)|^2d\theta
} 
corresponding to different ranges of $q$.   We will show that $\mci_1,\mci_2,\mci_3$ are bounded by the same bound as in Theorem \ref{minorthm}, which immediately implies Theorem \ref{minorthm}.  This section is to deal with $\mci_1$, and the next two sections deal with $\mci_2,\mci_3$ respectively.\\
 
 First we re-order the sum in $\hrnu$ according to the $u$ variable:
 \al{\label{ruf}
 \nonumber \hrnu(\theta)&=\sum_{x,y\in\ZZ}\sum_{\ff\in\FF}\sum_{u<U}\mu(u)\psi\left(\frac{x}{X}\right)\psi\left(\frac{2y}{X}\right)e(\theta\ff(x,2y))\\
 \nonumber&=\sum_{u\text{ odd}}\mu(u)\sum_{\ff\in\FF}\sum_{x,y\in\ZZ}\psi\left(\frac{xu}{X}\right)\psi\left(\frac{2yu}{X}\right)e(\theta\ff(xu,2yu))\\
 \nonumber&+\sum_{u\text{ even}}\mu(u)\sum_{\ff\in\FF}\sum_{x,y\in\ZZ}\psi\left(\frac{xu}{X}\right)\psi\left(\frac{yu}{X}\right)e(\theta\ff(xu,yu))\\
 &:=\sum_{u<U}\mu(u)\sum_{\ff\in\FF}\mathcal{R}_{u,\ff}(\theta)
 }

For simplicity we restrict our attention to $u$ even.  The same argument is applied to $u$ odd. We write $\frac{u^2}{q}=\frac{u_0}{q_0}$ in irreducible form, then we have
\al{\nonumber
\ruf\left(\frac{r}{q}+\beta\right)&=\sum_{x,y\in\ZZ}\mu(u)\psi\left(\frac{xu}{X}\right)\psi\left(\frac{yu}{X}\right)e\left(\ff(xu,yu)\left(\frac{r}{q}+\beta\right)\right)\\
\nonumber&=e\left(-b_{\ff}\left(\frac{r}{q}+\beta\right)\right)\sum_{x_0,y_0(q_0)}e\left(\frac{u_0}{q_0}\fff(x_0,y_0)r\right)\\
&\times\left[\sum_{x\equiv x_0(q_0),y\equiv y_0(q_0)}\psi\left(\frac{xu}{X}\right)\psi\left(\frac{yu}{X}\right)e(\fff(xu,yu)\beta)\right]
}

Now applying Poisson summation to the bracket, we have
\al{\label{ee}
\nonumber[\cdot]&=\sum_{\xi,\zeta\in\ZZ}\int_{-\infty}^{\infty}\infint\psi\left(\frac{(xq_0+x_0)u}{X}\right)\psi\left(\frac{(yq_0+y_0)u}{X}\right)e\left(\beta\fff((x_0+xq_0)u,(y_0+yq_0)u)-x\xi-y\zeta\right)dxdy\\
&=\frac{X^2}{u^2q_0^2}\sum_{\xi,\zeta\in\ZZ}e\left(\frac{x_0\xi}{q_0}+\frac{y_0\zeta}{q_0}\right)\dinfint\psi(x)\psi(y)e\left(\fff(x,y)X^2\beta-\frac{X\xi}{uq_0}x-\frac{X\zeta}{uq_0}y\right)dxdy
}

Putting \eqref{ee} back to \eqref{ruf}, we have
 \aln{\ruf\left(\frac{r}{q}+\beta\right)=\frac{X^2}{u^2}e\left(-b_{\ff}\left(\frac{r}{q}+\beta\right)\right)\sum_{\xi,\zeta\in\ZZ}\mathcal{S}_{\ff}(q_0,u_0r,\xi,\zeta)\mathcal{J}_{\ff}(\beta;uq_0,\xi,\zeta),
}

where 
\aln{
\mathcal{S}_{\ff}(q_0,u_0r,\xi,\zeta):=\frac{1}{q_0^2}\sum_{x_0,y_0(q_0)}e\left(\frac{u_0r}{q_0}\fff(x_0,y_0)+\frac{x_0\xi}{q_0}+\frac{y_0\zeta}{q_0}\right),
}
and 
\aln{
\mathcal{J}_{\ff}(\beta;uq_0,\xi,\zeta):=\dinfint\psi(x)\psi(y)e\left(\fff(x,y)X^2\beta-\frac{X\xi}{uq_0}x-\frac{X\zeta}{uq_0}y\right)dxdy.
}

We can compute $\mathcal{S}_{\ff}$ explicitly.  For simplicity we assume $q_0$ is odd, and $A_{\ff}$ is invertible in $
\ZZ/q_0\ZZ$.  We record a standard fact of exponential sum:
\aln{\sum_{a\in\ZZ/q\ZZ}e_q(x^2)=i^{\epsilon(q)}q^{\frac{1}{2}},
}
where $\epsilon(q)=0$ if $q\equiv 1(4)$ and $\epsilon(q)=1$ if $q\equiv 3(4)$.  From this, one can get 
\al{\label{exsum}\sum_{r\in\ZZ/q\ZZ}e_q(rx^2)=\left(\frac{r}{q}\right)i^{\epsilon(q)}q^{\frac{1}{2}}
}
if $(r,q)=1$.
Now complete square of $\mathcal{S}_{\ff}$ and apply \eqref{exsum} to $\mathcal{S}_{\ff}$, we get
\al{\nonumber
\mathcal{S}_{\ff}(q_0,u_0r,\xi,\zeta)&=\frac{1}{q_0^2}\sum_{x_0,y_0(q_0)}e_{q_0}\left(u_0r\fff(x_0,y_0)+x_0\xi+y_0\zeta\right)\\
\nonumber&=\frac{1}{q_0^2}\sum_{x_0,y_0(q_0)}e\left(u_0rA_{\ff}\left(x_0+B_{\ff}\bar{A}_{\ff}y_0\right)^2+\xi\left(x_0+B_{\ff}\bar{A}_{\ff}y_0\right)+2u_0r\bar{A}_{\ff}b_{\ff}^2y_0^2+\left(\zeta-\xi B_{\ff}\bar{A}_{\ff}\right)y_0\right)\\
&=\frac{1}{q_0^2}q_0^{\frac{1}{2}}i^{\epsilon(q_0)}\left(\frac{u_0rA_{\ff}}{q_0}\right)e_{q_0}\left(-\overline{4u_0rA_{\ff}}\xi^2\right)\sum_{y_0(q_0)}e_{q_0}\left(2u_0r\bar{A}_{\ff}b_{\ff}^2y_0^2+\left(\zeta-\xi B_{\ff}\bar{A}_{\ff}\right)y_0\right)
}

To deal with the sum in the above expression, we write $\frac{b_{\ff}^2}{q_0}=\frac{b_1}{q_1}$ where $(b_1,q_1)=1$.  Then after a linear change of variables and completing square we obtain
\al{\label{sf}\nonumber
\mathcal{S}_{\ff}(q_0,u_0r,\xi,\zeta)=\frac{i^{\epsilon(q_0)+\epsilon(q_1)}}{q_0^{\frac{1}{2}}q_1^{\frac{1}{2}}}\bd{1}_{\{A_{\ff}\zeta\equiv B_{\ff}\xi(\frac{q_0}{q_1})\}}\left(\frac{u_0rA_{\ff}}{q_0}\right)\left(\frac{2u_0rb_{\ff}\bar{A}_{\ff}}{q_1}\right)\\
\times e_{q_0}\left(-\overline{4u_0rA_{\ff}}\xi^2\right)e_{q_1}\left(-\overline{8u_0rb_1A_{\ff}}\left(\frac{q_1(A_{\ff}\zeta-B_{\ff}\xi)}{q_0}\right)^2\right)
}
From \eqref{sf} we see trivially that $|\mathcal{S}_{\ff}(q_0,u_0r,\xi,\zeta)|\leq q_0^{-\frac{1}{2}}$.  

Now we deal with $\mathcal{J}_{\ff}$.  For this we need standard results from non-stationary phase and stationary phase, and we record them here.\\

\emph{Non-stationary phase}: Let $\phi$ be a smooth compactly supported function on $(-\infty,\infty)$ and $f$ be a function which satisfies $|f^{'}(x)|>A>0$ in the support of $\phi$ and $A\geq |f^{(2)}(x)|,...,f^{(n)}(x)$ in the support of $\phi$.  Then
$$\infint\phi(x)e(f(x))dx\ll_{\phi,N}A^{-N}$$
\begin{proof}
By partial integration, 
\aln{&\infint\phi(x)e(f(x))=\infint\frac{\phi(x)}{f^{'}(x)}de(f(x))\\
&=-\infint\left(\frac{\phi}{f^{'}}\right)^{'}(x)e(f(x))dx=-\infint\frac{\phi^{'}(x)}{f^{'}(x)}+\frac{\phi(x)f^{(2)}(x)}{(f^{'}(x))^2}dx
}
From here, we see already that 
$$\infint\phi(x)e(f(x))dx\ll_{\phi,N}A^{-1}$$
Iterating partial integration $N$ times we can get the $A^{-N}$ bound.
\end{proof} 

\emph{Stationary phase}:  Let $f$ be a quadratic polynomial of two variables $x$ and $y$ with discriminant $-D$, where $D>0$.  Let $\phi(x,y)$ be a smooth compactly supported function on $\RR^2$, then 
$$\dinfint\phi(x,y)e(f(x,y))dxdy\ll_{\phi}\frac{1}{\sqrt{D}}.$$
\begin{proof}
After using an orthonormal matrix $L$ to change variables we can change the above integral into the from
$$\dinfint\phi(L(x,y))e\left(-x^2-\frac{D}{4}y^2\right)dxdy$$
Using Plancherel formula, 
\aln{
\dinfint\phi(L(x,y))e\left(-x^2-\frac{D}{4}y^2\right)dxdy=\frac{1}{i\sqrt{D}}\dinfint\widehat{\phi\circ L}(u,v)e\left(\frac{u^2}{4}+\frac{v^2}{D}\right)dudv\\
}
We caution the reader that $e(-x^2-\frac{D}{4}y^2)$ is not in $L^{2}$,  the above formula is obtained in the following way: first approximate $e^{2\pi i(-x^2-\frac{D}{4})y^2}$ by $e^{(-\epsilon+2\pi i)(-x^2-\frac{D}{r}y^2)}$, where we can apply Plancherel formula, then let $\epsilon\rightarrow 0$ and pass the limit.  Therefore, 
\al{\dinfint\phi(x,y)e(f(x,y))dxdy\leq\frac{1}{\sqrt{D}}||\widehat{\phi\circ L}||_1\leq\frac{1}{\sqrt{D}}||\phi||_1
}

\end{proof}
If either $\xi \gg U\geq TX\beta uq_0$ or $\zeta \gg U\geq TX\beta uq_0$, then the non-stationary phase condition is satisfied, we have $\mathcal{J}_{\ff}(\beta;uq_0,\xi,\zeta)\ll(\frac{uq_0}{X\xi})^{N}$ for any $N$, so these terms are negligible.  Now we deal with the case $\xi,\zeta\ll U$.  Recall that the discriminant of $\ff$ is $-8b_{\ff}^2$, by the stationary phase, we have
\al{\label{jbound}\mathcal{J}_{\ff}(\beta;uq_0,\xi,\zeta)\ll\text{min}\left\{1,\frac{1}{TX^2|\beta|}\right\}
}

With this, one gets
$$\ruf(\frac{r}{q}+\beta)\ll \frac{X^2}{u^2}\sum_{\xi,\zeta\ll u}q_0^{-\frac{1}{2}}\frac{1}{TX^2|\beta|}\ll\frac{u}{q^{\frac{1}{2}}T|\beta|}$$
using the fact that $u^2q_0\geq q$.  Therefore, we have 
\al{\label{rnubound}
\rnu\left(\frac{r}{q}+\beta\right)\ll T^{\delta}\sum_{u<U}\frac{u}{q^{\frac{1}{2}}T|\beta|}\ll\frac{T^{\delta-1}U^2}{q^{\frac{1}{2}}|\beta|}
}
Now we are able to bound $\mathcal{I}_1$
\begin{lem}\label{i1}
$$\mathcal{I}_{1}\ll NT^{2(\delta-1)}N^{-\eta}$$
\end{lem}
\begin{proof}
We divide the integral into three parts:
\aln{
\mathcal{I}_1&=\sum_{q<Q_0}\sum_{r(q)}{}^{'}\int_{\frac{r}{q}-\frac{1}{qM}}^{\frac{r}{q}+\frac{1}{qM}}\left|(1-\Tf(\theta))\hrnu(\theta)\right|^2d\theta\\&=\sum_{q<Q_0}\sum_{r(q)}^{'}\int_{-\frac{K_0}{N}}^{\frac{K_0}{N}}|\cdot|^2d\beta+\int_{\frac{K_0}{N}}{}^{\frac{1}{qM}}|\cdot|^2d\beta+\int_{-\frac{1}{qM}}^{-\frac{K_0}{N}}|\cdot|^2d\beta
}
For the first summand, we insert $|1-\Tf(\frac{r}{q}+\beta)|^2=\frac{N^2\beta^2}{K_0^2}$ and bound $\hrnu$ by \eqref{rnubound}.  For the second and the third summands,  we trivially bound $|1-\Tf(\theta)|^2$ by 1 and $\hrnu$
by \eqref{rnubound}.  Then we get
\al{\mathcal{I}_1\ll\frac{NQ_0T^{2(\delta-1)}U^4}{K_0}\llet T^{2\delta-1}X^2N^{-\eta},
}
which is a power saving.
\end{proof}

\subsection{Minor Arc Analysis II}  % minor arc analysis 2
In this section we deal with $\mathcal{I}_2$.  We divide the $q$-sum 2-adically: 
\al{\label{iq}\mathcal{I}_Q:=\sum_{Q\leq q<2Q}\sum_{r(q)}{}^{'}\int_{\frac{r}{q}-\frac{1}{qM}}^{\frac{r}{q}+\frac{1}{qM}}\left|\hrnu(\theta)\right|^2d\theta
}
We will show that for all $Q_0\leq Q<X$, $\II_{Q}$ has a power saving,  in the next section we will show that $\II_{Q}$ has a power saving for the range $X\leq Q\leq M$.  Clearly these will imply Theorem \ref{minorthm}.\\

Recall from \eqref{ruf} that 
\al{\nonumber
\hrnu\left(\frac{r}{q}+\beta\right)&=\sum_{u<U}\sum_{\ff\in\FF}\ruf\left(\frac{r}{q}+\beta\right)\\
&=\sum_{u<U}\sum_{\ff\in\FF}e\left(-b_{\ff}\left(\frac{r}{q}+\beta\right)\right)\frac{X^2}{u^2}\sum_{\xi,\zeta\in\ZZ}\mathcal{S}_{\ff}(q_0,u_0r,\xi,\zeta)\JJ_{\ff}(\beta;uq_0,\xi,\zeta)
}

Apply Cauchy-Schwartz inequality to the $u$ variable, we have
\al{\label{rnurq}
\nonumber\left|\hrnu\left(\frac{r}{q}+\beta\right)\right|^2\leq &X^4\left|\sum_{\ff\in\FF}\sum_{\xi,\zeta\in\ZZ}e\left(-b_{\ff}\left(\frac{r}{q}+\beta\right)\right)\mathcal{S}_{\ff}(q_0,u_0r,\xi,\zeta)\JJ_{\ff}(\beta;uq_0,\xi,\zeta)\right|^2\\
\nonumber=&X^4\sum_{\ff,\ff^{'}\in\FF}e\left(-\left(b_{\ff}-b_{\ff^{'}}\right)\frac{r}{q}\right)\sum_{\xi,\zeta\in\ZZ}\sum_{\xi^{'},\zeta^{'}\in\ZZ}\SM_{\ff}(q_0,u_0r,\xi,\zeta)\overline{\SM_{\ff^{'}}(q_0,u_0r,\xi^{'},\zeta^{'})}\\
\nonumber&\text{     }\JJ_{\ff}(\beta;uq_0,\xi,\zeta)\overline{\JJ_{\ff^{'}}(\beta;uq_0,\xi^{'},\zeta^{'})}e\left(-(b_{\ff}-b_{\ff^{'}})\beta\right)
}

Changing variables $\theta=\frac{r}{q}+\beta$ in \eqref{iq} and putting \eqref{rnurq} back to \eqref{iq}, we get
\al{\nonumber
\mathcal{I}_Q&\ll {X^4}\sum_{\ff,\ff^{'}\in\FF}\sum_{\xi,\zeta\in\ZZ}\sum_{\xi^{'},\zeta^{'}\in\ZZ}\sum_{Q\leq q<2Q}\left(\sum_{r(q)}{}^{'}e\left(-\left(b_{\ff}-b_{\ff^{'}}\right)\frac{r}{q}\right)\SM_{\ff}(q_0,u_0r,\xi,\zeta)\overline{\SM_{\ff^{'}}(q_0,u_0r,\xi^{'},\zeta^{'})}\right)\\
&\times \int_{-\frac{1}{qM}}^{\frac{1}{qM}}\JJ_{\ff}(\beta,uq_0,\xi,\zeta)\overline{\JJ_{\ff^{'}}(\beta;uq_0,\xi^{'},\zeta^{'})}e((-b_{\ff}+b_{\ff})\beta)d\beta}

We again split $\II_{Q}$ into non-Archimedean and Archimedean pieces.  For the Archimedean part, we use \eqref{jbound} to bound $\JJ$.  We have
\al{\nonumber&
\int_{-\frac{1}{qM}}^{\frac{1}{qM}}\JJ_{\ff}(\beta,uq_0,\xi,\zeta)\overline{\JJ_{\ff^{'}}(\beta;uq_0,\xi^{'},\zeta^{'})}e((-b_{\ff}+b_{\ff})\beta)d\beta\ll\infint \text{min}\left\{1,\frac{1}{TX^2|\beta|}\right\}^2d\beta\\
&\ll\int_{-\frac{1}{TX^2}}^{\frac{1}{TX^2}}1d\beta+\left(\int_{-\infty}^{-\frac{1}{TX^2}}+\int_{\frac{1}{TX^2}}^{\infty}\right)\frac{1}{T^2X^4\beta^2}d\beta\ll\frac{1}{TX^2}.
}
Now we analyze the non-Archimedean part.  Again for simplicity we only deal with $q_0$ odd, and $A_{\ff}$, $A_{\ff^{'}}$ invertible in $\ZZ/q_0\ZZ$.
We set 
\al{\label{ssf}\SM(q,q_0,u_0,\xi,\zeta,\ff,\xi^{'},\zeta^{'},\ff^{'})=\sum_{r(q)}{}^{'}e\left(-\left(b_{\ff}-b_{\ff^{'}}\right)\frac{r}{q}\right)\SM_{\ff}(q_0,u_0r,\xi,\zeta)\overline{\SM_{\ff^{'}}(q_0,u_0r,\xi^{'},\zeta^{'})}
}

Recall that $\frac{b_{\ff}^{2}}{q_0}=\frac{b_{1}}{q_1}$, and similarly we write $\frac{b_{\ff^{'}}^{2}}{q_0}=\frac{b_{1}^{'}}{q_1^{'}}$.  Plug \eqref{sf} in\eqref{ssf} , then we obtain
\al{\label{sd}\nonumber &
\SM(q,q_0,u_0,\xi,\zeta,\ff,\xi^{'},\zeta^{'},\ff^{'})=\bd{1}_{\sk{A_{\ff}\zeta\equiv B_{\ff}\xi(\frac{q_0}{q_1})\\A_{\ff^{'}}\zeta^{'}\equiv B_{\ff^{'}}\xi^{'}(\frac{q_0}{q_1^{'}})}}\times\frac{i^{\epsilon({q_1})-\epsilon(q_1^{'})}}{q_0q_1^{\frac{1}{2}}{q_1^{'}}^{\frac{1}{2}}}\left(\frac{A_{\ff}}{q_0}\right)\left(\frac{A_{\ff^{'}}}{q_0}\right)\\
&\nonumber\times \sum_{r(q)}{}^{'}\left(\frac{2u_0rb_1\bar{A}_{\ff}}{q_1}\right)\left(\frac{2u_0rb_1^{'}\bar{A}_{\ff^{'}}}{q_1^{'}}\right)e_{q}((-b_{\ff}+b_{\ff^{'}})r)e_{q_0}\left(-\overline{8u_0b_1A_{\ff}} \frac{q_0}{q_1}\left(\frac{q_1(A_{\ff}\zeta-B_{\ff}\xi)}{q_0}\right)^2\bar{r}\right)\\
&\times e_{q_0}\left(\overline{8u_0b_1^{'}A_{\ff^{'}}} \frac{q_0}{q_1^{'}}\left(\frac{q_1^{'}(A_{\ff^{'}}\zeta^{'}-B_{\ff^{'}}\xi^{'})}{q_0}\right)^2\bar{r}\right)e_{q_0}\left(-\overline{4u_0rA_{\ff}}\xi^2+\overline{4u_0rA_{\ff^{'}}}{\xi^{'}}^2\right)
}

This is a type of Kloosterman sum.  For our use in \eqref{sd} we only need an elementary $\frac{3}{4}$ bound originally due to Kloosterman \cite{Kl27} (compared to the $\frac{1}{2}$ bound implied by the Weil conjecture).  We stated it here:

\begin{lem}\label{Kloostermansum}
Let $S(m,n,q,\chi)=\sum_{x(q)}^{'}e_q(mx+n\bar{x})\chi(x)$, then we have
$$S(m,n,q,\chi)\ll_{\epsilon}\min\left\{(m,q),(n,q)\right\}^{\frac{1}{4}}q^{\frac{3}{4}+\epsilon}.$$
\end{lem}

We have an extra multiplicative character $\chi$ compared to the original paper by Kloosterman \cite{Kl27}, but his proof is easily modified to suit our case.\\

Apply Lemma \ref{Kloostermansum} to \eqref{sd}, and recall that $q_1=\frac{q_0}{(q_0,b_{\ff}^2)},q_1^{'}=\frac{q_0}{(q_0,b_{\ff^{'}}^2)}$,  then we obtain

\al{\label{twistsbound}|\mathcal{S}(q,q_0,u_0,\xi,\zeta,\ff,\xi^{'},\zeta^{'},\ff^{'})|\ll_{\epsilon}\left(b_{\ff}-b_{\ff^{'}},q\right)^{\frac{1}{4}}\left(\frac{q}{q_0}\right)^2(q_0,b_{\ff}^2)^{\frac{1}{2}}(q_0,{b_{{\ff^{'}}}}^2)^{\frac{1}{2}}q^{-\frac{5}{4}+\epsilon}.
}
In the case when $b_{\ff}=b_{\ff^{'}}$ and $\ff(\xi,-\zeta)\neq \ff^{'}(\xi^{'},-\zeta^{'})$, we prove a better bound for\\$\stwist$.  This will be needed in the next section.
\begin{lem} \label{betterbound}If $\bff=\bfp$, then
$$|\stwist|\ll_{\epsilon}(q_0,b_{\ff}^2)q^{-\frac{9}{8}+\epsilon}\left(\frac{q}{q_0}\right)^{\frac{17}{8}}\left|\ff(\xi,-\zeta)-\ff^{'}(\xi^{'},-\zeta^{'})\right|^{\frac{1}{2}}$$.
\end{lem}
\begin{proof}
If $b_{\ff}=b_{\ff^{'}}$, then $q_1=q_1^{'}=\frac{q_0}{(q_0,b_{\ff}^2)}$. From \eqref{sd} we have
\al{
\nonumber&|\stwist|=\frac{(q_0,\bff^2)}{q_0^2}\cdot\frac{q}{q_0}\biggr\rvert\sum_{r(q_0)}{}^{'}\left(\frac{A_{\ff}A_{\ff^{'}}}{q_1}\right)\\
\nonumber&\times e_{q_0}\left(-\overline{8u_0b_1A_{\ff}}\frac{q_0}{q_1}\left(\frac{q_1(A_{\ff}\zeta-B_{\ff}\xi)}{q_0}\right)^2\bar{r}\right)\times 
e_{q_0}\left(\overline{8u_0b_1^{'}A_{\ff^{'}}}\frac{q_0}{q_1^{'}}\left(\frac{q_1^{'}(A_{\ff^{'}}\zeta^{'}-B_{\ff^{'}}\xi^{'})}{q_0}\right)^2\bar{r}\right)\\
\nonumber &\times e_{q_0}(-\overline{4u_0rA_{\ff}}\xi^2+\overline{4u_0rA_{\ff^{'}}{\xi^{'}}^2})\biggr\rvert
}
Clearly the term $|\cdot|$ is multiplicative.  We apply the Kloostrman $3/4$ bound to $|\cdot|$ using the $\bar{r}$ coefficient:
\al{\label{pest}\nonumber
&|\stwist|\ll_{\epsilon}\frac{(q_0,b_{\ff}^2)}{q_0^2}\left(\frac{q}{q_0}\right)q_0^{\frac{3}{4}+\epsilon}\\&\times\prod_{p^{j}||q_0}\left(p^{j},-\bar{A_{\ff}}\xi^2-\overline{2b_1A_{\ff}}\frac{q_0}{q_1}L^2+\bar{A_{\ff^{'}}}{\xi^{'}}^2+\overline{2b_1A_{\ff^{'}}}\frac{q_0}{q_1}{L^{'}}^2\right)^{\frac{1}{4}}
}
where $L=\frac{q_1(A_{\ff}\xi-B_{\ff}\zeta)}{q_0}$ and $L^{'}=\frac{q_1^{'}(A_{\ff^{'}}\xi^{'}-B_{\ff^{'}}\zeta^{'})}{q_0}$. Now we divide the set of all the primes dividing $q_0$ into two sets $\mathcal{P}_1$ and $\mathcal{P}_2$, where $\PC_1$ contains primes $p$ such that
$$\bar{A_{\ff}}\xi^2+\overline{2b_1A_{\ff}}\frac{q_0}{q_1}L^2\equiv\bar{A_{\ff^{'}}}{\xi^{'}}^2+\overline{2b_1A_{\ff^{'}}}\frac{q_0}{q_1}{L^{'}}^2(p^{[j/2]})$$
and $\PC_2$ is the complement of $\PC_1$. \\

For $p\in\PC_2$, the gcd of $p^{j}$ and $-\bar{A_{\ff}}\xi^2-\overline{2b_1A_{\ff}}\frac{q_0}{q_1}L^2+\bar{A_{\ff^{'}}}{\xi^{'}}^2+\overline{2b_1A_{\ff^{'}}}\frac{q_0}{q_1}{L^{'}}^2$ is at most $p^{\frac{j}{2}}$.  Therefore,
\al{\label{psmall}
\prod_{p\in\PC_2}\left(p^{j},-\bar{A_{\ff}}\xi^2-\overline{2b_1A_{\ff}}\frac{q_0}{q_1}L^2+\bar{A_{\ff^{'}}}{\xi^{'}}^2+\overline{2b_1A_{\ff^{'}}}\frac{q_0}{q_1}{L^{'}}^2\right)\leq \prod_{p\in\PC_2}p^{\frac{j}{2}}\leq q_0^{\frac{1}{2}}
}

For $p\in\PC_1$, we have
\al{
\bar{A_{\ff}}\xi^{2}+\overline{2b_1A_{\ff}}\frac{q_0}{q_1}L^2\equiv\bar{A_{\ff}}\xi^2+\overline{2b_{\ff}^2}(A_{\ff}\zeta-B_{\ff}\xi)^2\equiv\overline{2b_{\ff}^2}\tilde{\ff}(\xi,-\zeta)(\text{mod }p^{[\frac{j}{2}]}).
}

Similarly,
\al{\bar{A_{\ff^{'}}}{\xi^{'}}^{2}+\overline{2b_1A_{\ff^{'}}}\frac{q_0}{q_1}L^2\equiv\overline{2b_{\ff^{'}}^2}\tilde{\ff^{'}}(\xi^{'},-\zeta^{'})(\text{mod }p^{[\frac{j}{2}]})}

Since $\bff=\bfp$, we have $\ff(\xi,-\zeta)\equiv\ff^{'}(\xi^{'},-\zeta^{'})$(mod $p^{\frac{j}{2}}$) for every $p\in\PC_1$. Thus we have

\al{\label{plarge}\prod_{p\in\PC_1}\left(p^{j},-\bar{A_{\ff}}\xi^2-\overline{2b_1A_{\ff}}\frac{q_0}{q_1}L^2+\bar{A_{\ff^{'}}}{\xi^{'}}^2+\overline{2b_1A_{\ff^{'}}}\frac{q_0}{q_1}{L^{'}}^2\right)\leq \prod_{p\in\PC_1}p^j\ll|\ff(\xi,-\zeta)-\ff^{'}(\xi^{'},-\zeta^{'})|^2
}

Plugging \eqref{psmall} and \eqref{plarge} back into \eqref{pest} we obtain our lemma.
 \end{proof}

Now we go back to $\II_Q$.  Again by non-stationary phase the sum is supported on the terms $\xi,\xi^{'},\zeta,\zeta^{'}\ll U$.  Using \eqref{twistsbound} we have
\al{
\nonumber&\II_Q\ll_{\epsilon}\frac{N^{\epsilon}X^4U^4}{TX^2}\sum_{\ff,\ff^{'}\in\FF}\sum_{Q\leq q\leq2Q}(b_{\ff}-b_{\ff^{'}},q)^{\frac{1}{4}}\left(\frac{q}{q_0}\right)^{2}(q_0,b_{\ff}^2)^{\frac{1}{2}}(q_0,\bfp^2)^{\frac{1}{2}}q^{-\frac{5}{4}}\\
&\label{iq25}\ll_{\epsilon}\frac{N^{\epsilon}X^2U^8}{T}\sum_{\ff,\ff^{'}\in\FF}\sum_{Q\leq q\leq2Q}(b_{\ff}-b_{\ff^{'}},q)^{\frac{1}{4}}(q_0,b_{\ff}^2)^{\frac{1}{2}}(q_0,\bfp^2)^{\frac{1}{2}}q^{-\frac{5}{4}}
}  

We further split \eqref{iq25} into two parts according to $b_{\ff}=b_{\ff^{'}}$ or not:
$$\II_{Q}\leq \II_{Q}^{(=)}+\II_Q^{(\neq)}.$$
We first deal with $\II_Q^{(=)}$.  Noticing that $\frac{q}{q_0}\leq U$, we have
\al{\label{iqe}
\II_Q^{(=)}\nonumber&\ll_{\epsilon}\frac{N^{\epsilon}X^2U^8}{T}\sum_{\ff\in\FF}\sum_{Q\leq q\leq2Q}\frac{(q,b_{\ff}^2)}{q}\sum_{\sk{\ff^{'}\in\FF\\b_{\ff^{'}}=b_{\ff}}}1\\
\nonumber&\ll_{\epsilon}\frac{N^{\epsilon}X^2U^8}{T}\sum_{\ff\in\FF}\sum_{a|b_{\ff}^2}a\sum_{Q\leq q\leq2Q}\bd{1}_{a|q}\sum_{\sk{\ff^{'}\in\FF\\b_{\ff^{'}}=b_{\ff}}}1\\
&\ll_{\epsilon}\frac{N^{\epsilon}X^2U^8}{T}\sum_{\ff\in\FF}\sum_{\sk{\ff^{'}\in\FF\\b_{\ff^{'}}=b_{\ff}}}1
}

For the last sum above, we introduce Lemma 5.2 from \cite{BK13}:

\begin{lem}[Bourgain, Kontorovich]\label{bk2}
There exists a positive constant $\mathcal{C}$ and there exists some $\eta_0 >0$ which only depend on the spectral gap of $\Gamma$ such that for any $1\leq q<N$ and any $r(${mod} $q$),
\aln{\sum_{\gamma\in\FF}\bd{1}_{{\langle e_1,\gamma\bd{r}\rangle\equiv r{(\text{mod }q)}}}\ll\frac{T^{\delta}}{q^{\eta_0}}.}
The implied constant is independent of $r$.
\end{lem}

Now we can finally determine $K_0,Q_0$ and $U$.  We set
\al{\label{determine}
Q_0=T^{\frac{\delta-\Theta}{20}}, K_0=Q_0^2,U=Q_0^{\frac{\eta_0^2}{100}}.
}
Apply Lemma \ref{bk2} to \eqref{iqe}, then we get 
\al{\II_Q^{(=)}\lle N^{-\eta_0+\epsilon}T^{2\delta-1}X^2U^9\llet T^{2\delta-1}X^2N^{-\eta}}
which is a power saving.

Now we deal with $\II_Q^{(\neq)}$. We introduce a parameter $H$ which is small power of $N$.  We further split $\II_Q^{(\neq)}$ into $\II_Q^{\neq,>}+\II_Q^{(\neq,\leq)}$ according to $(b_{\ff},\bfp)>H$ or not.  We first handle big gcd.

\begin{lem} 
$$\II_Q^{(\neq,>)}\llet NT^{2(\delta-1)}N^{-\eta}$$
\end{lem}
\begin{proof}
Apply \eqref{iq25} and replace $(q_0,b_{\ff}^2)$ by $(q,b_{\ff}^2)$ and $(b_{\ff}-b_{\ff^{'}},q)$ by $q$:
\al{\nonumber
\II_Q^{(\neq,>)}&\lle \frac{N^{\epsilon}X^2U^8}{T}\sum_{\ff\in\FF}\sum_{\sk{\ff^{'}\in\FF\\(b_{\ff},b_{\ff^{'}})>H}}\sum_{Q\leq q\leq2Q}\frac{(q,b_{\ff}^2)^{\frac{1}{2}}(q,b_{\ff^{'}}^2)^{\frac{1}{2}}}{q}\\
&\nonumber\lle \frac{N^{\epsilon}X^2U^8}{T}\sum_{\ff\in\FF}\sum_{\sk{h|b_{\ff}^2\\h>H}}\sum_{\sk{\ff^{'}\in\FF\\b_{\ff^{'}}\equiv 0(h)}}\sum_{Q\leq q\leq2Q}\frac{(q,b_{\ff}^2)^{\frac{1}{2}}(q,b_{\ff^{'}}^2)^{\frac{1}{2}}}{q}\\
&\lle \frac{N^{\epsilon}X^2U^8}{T}\sum_{\ff\in\FF}\sum_{\sk{h|b_{\ff}^2\\h>H}}\sum_{\sk{\ff^{'}\in\FF\\b_{\ff^{'}}\equiv 0(h)}}\sum_{\tilde{q_1}|b_{\ff^2}}\sum_{\tilde{q_1}^{'}|\bfp^2}(\tilde{q_1}\tilde{q_1}^{'})^{\frac{1}{2}}\sum_{\substack{Q\leq q\leq 2Q \\ [\tilde{q_1},\tilde{q_1}^{'}]|q}}\bd{1}
}
Now since $[\tilde{q_1},\tilde{q_1}^{'}]>(\tilde{q_1}\tilde{q_1}^{'})^{\frac{1}{2}}$, the above 
\al{\label{gcdlarge}
\lle  \frac{N^{\epsilon}X^2U^9}{T}\sum_{\ff\in\FF}\sum_{\sk{h|b_{\ff}^2\\h>H}}\sum_{\sk{\ff^{'}\in\FF\\b_{\ff^{'}}\equiv 0(h)}}\sum_{\tilde{q_1}|b_{\ff^2}}\sum_{\tilde{q_1}^{'}|\bfp^2}1}

From Lemma \ref{bk2},  we have $\sum_{\sk{\ff^{'}\in\FF\\{\bfp\equiv0(h)}}}1\ll\frac{T^{\delta}}{H^{\eta_0}}.$ We set $H=Q_0^{\frac{\eta_0}{10}}$. Therefore,
\aln{
\eqref{gcdlarge}\lle \frac{\nep X^2U^9T^{2\delta}}{TH^{\eta_0}}\lle \frac{\nep T^{2\delta-1}X^2U^9}{H^{\eta_0}}\llet T^{2\delta-1}X^2N^{-\eta}
}
which is a power saving. 
\end{proof}
Next we deal with small gcd. We write $(b_{\ff},b_{\ff^{'}})=h$ and $b_{\ff}=hg_1,\bfp=hg_2,\bff-\bfp=hg_3$.  Then $g_1,g_2,g_3$ are mutually relatively prime.  We have
\begin{lem}
$$\iqs\llet N^{1-\eta}T^{2(\delta-1)}$$
\end{lem}
\begin{proof}
From \eqref{iq25},
\aln{
\iqs&\lle \frac{\nep X^2U^8}{T}\sum_{\ff\in\FF}\sum_{\sk{\ff^{'}\in\FF\\(\bfp,\bff)\leq H}}\sum_{Q\leq q\leq2Q} \frac{(q_0,\bff^2)^{\frac{1}{2}}(q_0,\bfp^2)^{\frac{1}{2}}(\bff-\bfp,q)^{\frac{1}{4}}}{q^{\frac{5}{4}}}\\
&\lle \frac{\nep X^2U^8}{TQ^{\frac{5}{4}}}\sum_{\ff\in\FF}\sum_{\sk{\ff^{'}\in\FF\\(\bfp,\bff)\leq H}}\sum_{Q\leq q\leq2Q}{(q_0,\bff)(q_0,\bfp)(\bff-\bfp,q)^{\frac{1}{4}}}\\
&\lle \frac{\nep X^2U^8}{TQ^{\frac{5}{4}}}\sum_{\ff\in\FF}\sum_{\sk{\ff^{'}\in\FF\\(\bfp,\bff)\leq H}}\sum_{h|(\bff,\bfp)}h^{\frac{9}{4}}\sum_{g_1|b_{\ff}}\sum_{g_2|\bfp}\sum_{\sk{g_3|\bff-\bfp\\g_3\ll Q}} g_1g_2g_3^{\frac{1}{4}}\sum_{\sk{Q\leq q\leq2Q\\ [hg_1,hg_2,hg_3]|q}}1\\
&\lle \frac{\nep X^2U^8H^{\frac{9}{4}}}{TQ^{\frac{5}{4}}}\sum_{\ff\in\FF}\sum_{\sk{\ffp\in\FF\\(\bfp,\bff)\leq H}}\sum_{g_1|b_{\ff}}\sum_{g_2|\bfp}\sum_{\sk{g_3|\bff-\bfp\\g_3\ll Q}} g_1g_2g_3^{\frac{1}{4}}\frac{q}{g_1g_2g_3}\\
&\lle  \frac{\nep X^2U^8H^{\frac{9}{4}}}{TQ^{\frac{1}{4}}}\sum_{\ff\in\FF}\sum_{\sk{\ffp\in\FF\\(\bfp,\bff)\leq H}}\sum_{\sk{g_3|\bff-\bfp\\g_3\ll Q}}g_3^{-\frac{3}{4}}\\
&\lle \frac{\nep X^2U^8H^{\frac{9}{4}}}{TQ^{\frac{1}{4}}}\sum_{\ff\in\FF}\sum_{g_3\ll Q}g_3^{-\frac{3}{4}}\sum_{\sk{\ff\in\FF\\\bfp\equiv \bff(g_3)}}1
}
By Lemma \ref{bk2},  
$\sum_{\sk{\ffp\in\FF}}\bd{1}_{\bfp\equiv \bff(g_3)}\ll \frac{T^{\delta}}{g_3^{\eta_0}}$.  Therefore,
$$\iqs\ll_{\epsilon}\frac{\nep X^2U^9H^{\frac{9}{4}}}{TQ^{\frac{1}{4}}}\sum_{\ff\in\FF}T^{\delta}Q^{\frac{1}{4}-\eta_0}\lle \nep T^{2\delta-1}X^2U^9H^{\frac{9}{4}}Q_0^{-\eta_0}\llet T^{2\delta-1}X^2N^{-\eta}$$

Again we have a power savings for $\mathcal{I}_Q^{(\neq,\leq)}$. 
\end{proof}
In summary, we have
\begin{lem}\label{i2}
$$\mathcal{I}_2\ll_{\eta}N^{1-\eta}T^{2(\delta-1)}$$
\end{lem}

\subsection{Minor Arc Analysis III}  %minor arc analysis III
In this section we deal with the last part of the integral, which is on the minor arcs corresponding to $X<q<M$, namely $\II_3$.  We keep all the notations from the previous sections.  Return to \eqref{ruf}, and again for simplicity we restrict our attention on the summands of $\rnu$ where $u$ even:
\al{&\nonumber \ruf\left(\frac{r}{q}+\beta\right)=\sum_{x,y\in\ZZ}\psi\left(\frac{xu}{X}\right)\psi\left(\frac{yu}{X}\right)e\left(\ff(xu,yu)\left(\frac{r}{q}+\beta\right)\right)\\
&=e\left(-\left(\frac{r}{q}+\beta\right)\bff\right)\sum_{x,y\in\ZZ}\psi\left(\frac{xu}{X}\right)\psi\left(\frac{yu}{X}\right)e\left(\frac{ru_0\fff(x,y)}{q_0}\right)e(\fff(x,y)u^2\beta)
}
Now we rewrite $e_{q_0}(ru_0\fff(x,y))$ into its Fourier expansion.  We have
\aln{
\nonumber e_{q_0}(ru_0\fff(x,y))&=\frac{1}{q_0^2}\sum_{m(q_0)}\sum_{n(q_0)}\sum_{l(q_0)}\sum_{t(q_0)}e_{q_0}(ru_0\fff(l,t)+lm+tn)e_{q_0}(-mx-ny)\\
&=\sum_{m(q_0)}\sum_{n(q_0)}\SM_{\ff}(q_0,u_0r,m,n)e_{q_0}(-mx-ny)
}
Therefore,
$$\ruf\left(\frac{r}{q}+\beta\right)=e_q(-r\bff)\sum_{m(q_0)}\sum_{n(q_0)}\SM_{\ff}(q_0,u_0r,m,n)\lambda_{\ff}\left(X,\beta;\frac{m}{q_0},\frac{n}{q_0},u\right),$$
where
$$\lambda_{\ff}\left(X,\beta;\frac{m}{q_0},\frac{n}{q_0},u\right):=\sum_{x,y\in\ZZ}\psi\left(\frac{xu}{X}\right)\psi\left(\frac{yu}{X}\right)e\left(-\frac{mx}{q_0}\right)e\left(-\frac{ny}{q_0}\right)e(\ff(xu,yu)\beta).$$

We apply the Cauchy-Schwarz inequality to the $u$ variable for $\II_Q$:
\aln{
\II_Q&=\sum_{Q\leq q \leq 2Q}\sum_{r(q)}{}^{'}\int_{-\frac{1}{qM}}^{\frac{1}{qM}}\left|\hrnu\left(\frac{r}{q}+\beta\right)\right|^2d\beta\\
&\ll U\sum_{u<U}\sum_{Q\leq q\leq 2Q}\sum_{r(q)}{}^{'}\int_{-\frac{1}{qM}}^{\frac{1}{qM}}\left|\sum_{\ff\in\FF}\ruf\left(\frac{r}{q}+\beta\right)\right|^2d\beta\\
&\ll U\sum_{u<U}\sum_{\ff\in\FF}\sum_{\ff^{'}\in\FF}\sum_{Q\leq q\leq2Q}\sum_{m,n,m^{'},n^{'}(q_0)}\left(\sum_{r(q)}{}^{'}\SM_{\ff}(q_0,u_0r,m,n)\overline{\SM_{\ff^{'}}(q_0,u_0r,m^{'},n^{'})}e_q(r(-\bff+\bfp))\right)\\
&\times \int_{-\frac{1}{qM}}^{\frac{1}{qM}}\lambda_{\ff}\left(X,\beta;\frac{m}{q_0},\frac{n}{q_0},u\right)\overline{\lambda_{\ff^{'}}\left(X,\beta;\frac{m^{'}}{q_0},\frac{n^{'}}{q_0},u\right)}d\beta
}

Since $m,n,m^{'},n^{'}$ comes from congruence classes (mod $q_0$), we can choose representatives such that $m,n,m^{'},n^{'}$ with absolute values bounded by $\frac{q_0}{2}$.  The main contribution of $\II_Q$ comes from the terms $m,n,m^{'},n^{'}\ll\frac{uq_0}{X}$ by non-stationary phase.  To see this,  for the terms with any of $m,n,m^{'}, n^{'}\gg\frac{uq_0}{X}$ (let's say $m\gg\frac{uq_0}{X}$), we use Poisson summation to rewrite $\lambda_{\ff}$:
\aln{
&\lambda_{\ff}\left(X,\beta;\frac{m}{q_0},\frac{n}{q_0},u\right)=\frac{X^2}{u^2}\dinfint\psi(x)\psi(y)e\left(-\frac{mX}{q_0u}x\right)e\left(-\frac{nX}{q_0u}y\right)e(\ff(xX,yX)\beta)dxdy\\
&+\sum_{\sk{\xi,\zeta\in\ZZ\\(\xi,\zeta)\neq(0,0)}}\frac{X^2}{u^2}\dinfint\psi(x)\psi(y)e\left(\left(\frac{\xi X}{u}-\frac{mX}{q_0u}\right)x\right)e\left(\left(\frac{\zeta X}{u}-\frac{nX}{q_0u}\right)y\right)e(\ff(xX,yX)\beta)dxdy
}
If $\xi\neq 0$, since $\frac{mX}{q_0u}\leq \frac{X}{2u}$ and $\ff(xX,yX)\beta\ll 1$, we have
\aln{
\dinfint\psi(x)\psi(y)e\left(\left(\frac{\xi X}{u}-\frac{mX}{q_0 u}\right)x\right)e\left(\left(\frac{\zeta X}{u}-\frac{nX}{q_0u}\right)y\right)e(\ff(xX,yX)\beta)dxdy\ll \left(\frac{u}{X\xi}\right)^{N_0}
}
for any $N_0>0$, by first applying non-stationary phase to the $x$ variable and trivially bounding the $y$ integral.  From this, one gets
\al{
\lambda_{\ff}\left(X,\beta;\frac{m}{q_0},\frac{n}{q_0},u\right)&=\frac{X^2}{u^2}\dinfint\psi(x)\psi(y)e\left(-\frac{mX}{q_0u}\right)e\left(-\frac{nX}{q_0u}\right)e(\ff(xX,yX)\beta)dxdy\\&+O\left(\left(\frac{u}{X}\right)^{N_0}\right)
} 
for any $N_0>0$.  We use non-stationary phase again to treat the above integral, then we obtain
$$\Big\rvert\lambda_{\ff}\left(X,\beta,\frac{m}{q_0},\frac{n}{q_0},u\right)\Big\rvert\ll\frac{X^2}{u^2}\text{min}\left\{\left(\frac{uq_0}{Xm}\right)^{2N_0},\left(\frac{uq_0}{Xn}\right)^{2N_0}\right\}\ll\frac{X^2}{u^2}\left(\frac{uq_0}{X}\right)^{2N_0}\frac{1}{m^{N_0}n^{N_0}}.$$

Therefore, we have
\aln{
\int_{-\frac{1}{qM}}^{\frac{1}{qM}}\lambda_{\ff}\left(X,\beta;\frac{m}{q_0},\frac{n}{q_0},u\right)\overline{\lambda_{\ff^{'}}\left(X,\beta;\frac{m^{'}}{q_0},\frac{n^{'}}{q_0},u\right)}d\beta\ll \frac{1}{QM}\frac{X^4}{u^4}\left(\frac{uq_0}{X}\right)^{4N_0}\frac{1}{m^{N_0}n^{N_0}{m^{'}}^{N_0}{n^{'}}^{N_0}}}
Now we use \eqref{twistsbound} to bound $|\mathcal{S}|$, we thus have
\al{
\nonumber&U\sum_{u<U}\sum_{\ff\in\FF}\sum_{\ff^{'}\in\FF}\sum_{Q\leq q\leq2Q}\sum_{\sk{m,n,m^{'},\text{ or }n^{'}\gg\frac{uq_0}{X}}}\left(\sum_{r(q)}^{'}\SM_{\ff}(q_0,u_0r,m,n)\SM_{\ff^{'}}(q_0,u_0r,m^{'},n^{'})e_q(r(-\bff+\bfp))\right)\\
\nonumber&\times \int_{-\frac{1}{qM}}^{\frac{1}{qM}}\lambda_{\ff}\left(X,\beta;\frac{m}{q_0},\frac{n}{q_0},u\right)\overline{\lambda_{\ff^{'}}\left(X,\beta;\frac{m^{'}}{q_0},\frac{n^{'}}{q_0},u\right)}d\beta\\
&\lle N^{\epsilon}UT^{2\delta}\sum_{u<U}\sum_{Q\leq q\leq2Q}\frac{T^{\frac{9}{4}}u^4}{Q^{\frac{5}{4}}}\frac{1}{QM}\frac{X^4}{u^4}\left(\frac{uq_0}{X}\right)^{4N_0}\sum_{m,n,m^{'},\text{or }n^{'}\gg \frac{uq_0}{X}}\frac{1}{m^{N_0}n^{N_0}{m^{'}}^{N_0}{n^{'}}^{N_0}}}
If We set $N_0=5$, then the above 
$$\ll N^{\epsilon}U^{20}T^{2\delta+\frac{63}{4}}X^{\frac{7}{4}}$$

Thus we see $|\mathcal{S}|$ is indeed mainly supported on $m,n,m^{'},n^{'}\ll \frac{uq_0}{X}$.   Now we split the terms $m,n,m^{'},n^{'}\ll \frac{uq_0}{X}$ into two parts according to whether $b_{\ff}=b_{\ff^{'}}$ or not:
$$\II_Q\ll\II_Q^{(=)}+\II_Q^{(\neq)},$$
where
\al{\nonumber&
\II_{Q}^{(=)}=\sum_{u<U}\sum_{\ff\in\FF}\sum_{\sk{\ff^{'}\in\FF\\\bfp=\bff}}\sum_{Q\leq q\leq2Q}\sum_{m,n,m^{'},n^{'}\ll \frac{uq_0}{X}}\SM(q,q_0,u_0,\xi,\zeta,\ff,\xi^{'},\zeta^{'},\ff^{'})\\
&\times\int_{-\frac{1}{qM}}^{\frac{1}{qM}}\lambda_{\ff}\left(X,\beta;\frac{m}{q_0},\frac{n}{q_0},u\right)\overline{\lambda_{\ff^{'}}\left(X,\beta;\frac{m^{'}}{q_0},\frac{n^{'}}{q_0},u\right)}d\beta
}
and 
\al{\nonumber&
\II_{Q}^{(\neq)}=\sum_{u<U}\sum_{\ff\in\FF}\sum_{\sk{\ff^{'}\in\FF\\\bfp\neq\bff}}\sum_{Q\leq q\leq2Q}\sum_{m,n,m^{'},n^{'}\ll \frac{uq_0}{X}}\SM(q,q_0,u_0,\xi,\zeta,\ff,\xi^{'},\zeta^{'},\ff^{'})\\
&\times\int_{-\frac{1}{qM}}^{\frac{1}{qM}}\lambda_{\ff}\left(X,\beta;\frac{m}{q_0},\frac{n}{q_0},u\right)\overline{\lambda_{\ff^{'}}\left(X,\beta;\frac{m^{'}}{q_0},\frac{n^{'}}{q_0},u\right)}d\beta
}
For $\lambda$, since the sum is supported on $x,y\asymp\frac{X}{u}$, $\lambda$ has a trivial bound $\frac{X^2}{u^2}$.  Therefore, for $\square\in\{=,\neq\}$, we have 
\al{\label{square}
\II_Q^{\square}\ll \frac{UX^4}{QM}\sum_{u<U}\frac{1}{u^4}\sum_{\ff\in\FF}\sum_{\sk{\ff^{'}\in\FF\\\bfp\square\bff}}\sum_{Q\leq q\leq2Q}\sum_{m,n,m^{'},n^{'}\ll \frac{uq_0}{X}}|\mathcal{S}|.
}
If $\bff\neq\bfp$, then we could use the bound from \eqref{twistsbound} to estimate $\SM$.  We have

\al{\label{sig1}
\nonumber\II_Q^{(\neq)}&\lle \frac{UX^4}{QM}\sum_{u<U}\frac{1}{u^4}\sum_{\ff\in\FF}\sum_{\sk{\ff^{'}\in\FF\\\bfp=\bff}}\sum_{Q\leq q\leq2Q}\sum_{m,n,m^{'},n^{'}}(\bff-\bfp,q)^{\frac{1}{4}}\left(\frac{q}{q_0}\right)^2{(q_0,\bff^2)^{\frac{1}{2}}(q_0,\bfp^2)^{\frac{1}{2}}}q^{-\frac{5}{4}+\epsilon}\\
&\lle \frac{N^{\epsilon}UX^4}{QM}\sum_{u<U}\frac{1}{u^4}\sum_{\ff\in\FF}\sum_{\sk{\ff^{'}\in\FF\\\bff\neq\bfp}}\sum_{Q\leq q\leq2Q}\left(\frac{uq_0}{X}\right)^4T^{\frac{9}{4}}u^4Q^{-\frac{5}{4}}\lle T^{2(\delta-1)}N^{1+\epsilon}(T^6X^{-\frac{1}{4}}U^6)
}
where we replaced $(\bff-\bfp,q),(q_0,\bff^2)^{\frac{1}{2}}$ and $(q_0,\bfp^2)^{\frac{1}{2}}$ by $T$.  Thus we have a significant power saving for $\II_Q^{(\neq)}$.  

Next we deal with $\II_Q^{(=)}$,  we further split $\II_Q^{(=)}$ into two pieces 
$$\II_Q^{(=)}=\II_Q^{(=,=)}+\II_Q^{(=,\neq)}$$
according to whether $\ff(m,-n)=\ff^{'}(m^{'},-n^{'})$ or not. For $\II_Q^{(=,\neq)}$, we use Lemma \ref{betterbound} to bound $|\SM|$.  We have
\al{\label{iqeqeq}
\II_Q^{(=,\neq)}\ll \frac{UX^4}{QM}\sum_{u<U}\sum_{Q\leq q\leq2Q}\sum_{m,n,m^{'},n^{'}\ll\frac{uq_0}{X}}\sum_{\sk{\ff,\ff^{'}\in\FF\\\bff=\bfp\\\ff(m,-n)\neq\ff^{'}(m^{'},-n^{'})}}(q_0,\bff^2)q^{-\frac{9}{8}+\epsilon}\left(\frac{q}{q_0}\right)^{\frac{17}{8}}\left|\ff(m,-n)-\ff^{'}(m^{'},-n^{'})\right|^{\frac{1}{2}}
}
Noticing that $(q_0,\bff^2)\ll T^2,\frac{q}{q_0}\ll U^2$ and $\ff(m,-n),\ff^{'}(m^{'},-n^{'})\ll T(\frac{UQ}{X})^2$, we have
\al{\label{sig2}\II_Q^{(=,\neq)}\lle \nep\frac{UX^4}{QM}UQ\left(\frac{UQ}{X}\right)^4T^{2\delta}\frac{T^2}{Q^{\frac{9}{8}}}U^{\frac{17}{4}}T^{\frac{1}{2}}\frac{UQ}{X}\lle \nep U^{\frac{45}{4}}T^{2\delta+\frac{43}{8}}X^{\frac{15}{8}},}

which is again a significant power saving.\\

Next we deal with $\II_Q^{(=,=)}$.  This will complete our minor arc analysis.   From \eqref{twistsbound} and \eqref{square} we have
$$\II_Q^{(=,=)}\ll\frac{UX^4}{QM}\sum_{u<U}\frac{1}{u^4}\sum_{\ff\in\FF}\sum_{Q\leq q\leq2Q}\sum_{m,n\ll\frac{uq_0}{X}}\frac{(\bff^2,q)}{q}u^4\sum_{\sk{\ff^{'}\in\FF\\\bfp=\bff}}\sum_{\sk{m^{'},n^{'}\ll \frac{uq_0}{X}\\ \ff^{'}(m^{'},-n^{'})=\ff(m,-n)}}1$$

For the inner double sum we shall prove the following lemma:
\begin{lem}\label{innerdsum}
$$\sum_{\sk{\ff^{'}\in\FF\\\bfp=\bff}}\sum_{\sk{m^{'},n^{'}\ll \frac{uq_0}{X}\\ \ff^{'}(m^{'},-n^{'})={\ff}(m,-n)}}1\lle N^{\epsilon}\left(\fff(m,-n),-8\bff^2\right)^{\frac{1}{2}} $$
\end{lem}
 \begin{proof}
 This lemma will follow from the following three claims.  \\
 
$\bd{Claim}$ 1:  The number of classes of equivalent quadratic forms having discriminant $-8b_{\ff}^2$ and representing the integer $z=\tilde{\ff}(m,-n)$ is bounded by $\nep(\tilde{\ff}(m,-n),-8\bff^2)^{\frac{1}{2}}$.\\

Suppose $z=\tilde{\ff}(m,-n)$ is primitively represented by a quadratic form $\ff_0$ (i.e. $(m,n)=1$), then $\ff_0$ is equivalent to a quadratic form $zx^2+B_0xy+C_0y^2$, with $|B_0|<z$.   Now since $B_0^2-4zC_0 =-8b_{\ff}^2$, we have 
$B_0^2\equiv -8\bff^2(z)$.  From the Chinese Remainder Theorem, the number of solutions of
\al{\label{zequation}B_0^2\equiv -8\bff^2(z)} is the the product of the numbers of solutions of 
\al{\label{pequation}B_0^2\equiv -8\bff^2(p_i^{n_i})} for each $p_i^{n_i}||z$.\\

If $\left(\frac{-2}{p_i^{n_i}}\right)=-1$, then there's no solution to \eqref{pequation}.  If $\left(\frac{-2}{p_i^{n_i}}\right)=1$, let $-8b_{\ff}^2\equiv kp_i^{l_i}(p_i^{n_i})$ where $0\leq l_i\leq n_i$ and $(k,p_i)=1$.  Noticing that $l_i$ is even, all the solutions of \eqref{pequation} are given by
$$\pm p^{\frac{l_i}{2}}l+p^{n_i-\frac{l_i}{2}}s,$$
where $l$ is a solution of 
\al{\label{lequation}l^2\equiv \frac{-8b_{\ff}^2}{p_i^{l_i}}(p^{n_i-l_i})}
and $0\leq s\leq p^{\frac{l_i}{2}}-1$.  Thus we see there are at most $2p^{\frac{l_i}{2}}$ such solutions to \eqref{pequation}.  By multiplicativity, the number of solutions of  \eqref{zequation} is bounded by $2^{w(\fff(m,-n))}(\fff(m,-n),-8\bff^2)^{\frac{1}{2}}$.  Therefore, our choices for $B_0$ is at most $2^{w(\tilde{\ff}(m,-n))+1}(\fff(m,-n),-8\bff^2)^{\frac{1}{2}}$. If $z$ is not primitively represented by $\ff_0$, then a divisor $z_0$ of $z$ is primitively represented.  There are at most $d(z)$ many such cases, and the bound $2^{w(\fff(m,-n))+1}(\fff(m,-n),-8\bff^2)^{\frac{1}{2}}$ works for each case.  Thus Claim 1 follows. \\

$\bd{Claim}$ 2: In each equivalent class in $\FF$, the number of equivalent quadratic forms is bounded:
Suppose $\ff^{'}=(A^{'},2B^{'},C^{'})$ and $\ff^{''}=(A^{''},2B^{''},C^{''})$ are two equivalent quadratic forms in $\FF$, then we can find $\mat{g&h\\i&j}\in SL(2,\ZZ)\cup \mat{1&0\\0&-1}SL(2,\ZZ)$ such that
\al{
\nonumber&A^{''}=g^2A^{'}+2giB^{'}+i^2C^{'},\\
\nonumber&B^{''}=ghA^{'}+(gi+hj)B^{'}+ijC^{'},\\
&C^{''}=h^2A^{'}+2hjB^{'}+j^2C^{'}
}
The first equation above can be rewritten as 
$A^{''}=A^{'}\left(g+i\frac{B^{'}}{A^{'}}\right)^2+i^2\frac{2\bff^2}{A^{'}}$.  So from $i^2\frac{2\bff^2}{A^{'}}\leq A^{''}$, $\bff\asymp T$, $A^{'},A^{''}\ll T$,  we know $i\ll 1$, and $A^{'},A^{''}\gg T$.  Then from $A^{'}\left(g+i\frac{B^{'}}{A^{'}}\right)^2\leq A^{''}\ll T $ we also know $g\ll1$.  Similarly $h,j\ll 1$,  so the number of quadratic forms in $\FF$ in each equivalent class  is bounded. Therefore Claim 2 holds.\\

 $\bd{Claim}$ 3: given an integer $z\ll N$ and a quadratic form $\ff$ of discriminant $-8\bff^2$, there are at most $N^{\epsilon}$ pairs of integers $m,n$ such that $\ff(m,-n)=z$.\\

This is because $Am^2-2Bmn+Cn^2=z$ can be rewritten as
$$(Am+(B+\sqrt{-2}\bff)n)(Am+(B-\sqrt{-2}\bff)n)=Az$$
Since $Az\ll N^2$, the number of divisors of $Az$ in $\ZZ[\sqrt{2}i]$ is bounded by $N^{\epsilon}$.  The pairs $(m,n)$ can be identified with $Am+(B+\sqrt{-2}\bff)n$, which is a divisor of $Az$. Therefore, Claim 3 also holds.\\

Our lemma then follows Claims 1, Claim 2 and Claim 3.
 \end{proof}
 
 We need the following final ingredient to estimate $\mathcal{I}_Q^{(=,=)}$:
 \begin{lem} \label{finalingrediant}Given a primitive quadratic form $(A,2B,C)$ of discriminant $-8b_{\ff}^2$, for any $d|2b_{\ff}^2$, and any integer $W>0$, we have
 $$\sum_{\sk{m,n\leq W\\Am^2-2Bmn+Cn^2\equiv 0(d)}}1\ll W^2d^{-\frac{1}{2}}+W$$
 The implied constant is absolute.
 \end{lem}
\begin{proof}
First we show that $\exists \gamma=\mat{i&j\\g&h}\in SL(2,\ZZ)$ and $\tilde{A},\tilde{B},\tilde{C}\in\ZZ$ such that 
$$Ax^2+2Bxy+Cy^2=\tilde{A}(ix+gy)^2+\tilde{B}(ij+gh)xy+\tilde{C}(jx+hy)^2$$
and $$(\tilde{A},-2\bff^2)=1,\tilde{B}\equiv\tilde{C}\equiv 0(d).$$
Indeed, for each $p_i^{n_i}||d$, since $\ff$ is primitive, at least one of $A,B,C$ can not be divided by $p$.  For example, if $(A,p)=1$, then
$$Ax^2+2Bxy+Cy^2\equiv A(x+B\bar{A}y)^2+2\bff^2\bar{A}y^2\equiv A(x+B\bar{A}y)^2.$$
We set $$\gamma_{p_i^{n_i}}:=\mat{1&B\bar{A}\\0&1}\in SL(2,\ZZ/p_i^{n_i}\ZZ)$$ so
$\gamma_{p_i^{n_i}}(A,2B,C)=(A,0,0)(p_i^{n_i})$.  Now from the Chinese remainder theorem, we could find $\gamma_d\in SL(2,\ZZ/d\ZZ)$ such that $\gamma_d\equiv \gamma_{p_i^{n_i}}$ in $SL(2,\ZZ/p_i^{n_i})$ for each $p_i^{n_i}||d$.  Since $(\tilde{A},d)=1$ and $\tilde{B}\equiv 0(d)$, it forces $\tilde{C}\equiv 0(d)$.  
Therefore,
$$\sum_{\sk{m,n\leq W\\\ff(m,-n)\equiv 0(d)}}1= \sum_{\sk{m,n\leq W\\(im+gn)^2\equiv 0(d)}}{1}$$
If $(im+gn)^2\equiv 0(d)$, then $im+gn$ can be parametrized by $sd_0$, where $s\in\ZZ$ and $d_0\geq d^{\frac{1}{2}}$.  Therefore, we have
\al{\label{final}im+gn\equiv 0(d_0)}
For the above equation to have a solution, since $(i,g)=1$, $gn$ should be of the form $k(i,d_0)$ where $k\in\ZZ$, so there are at most $\frac{W}{(i,d_0)}+1$ choices for $n$.  Fixing such an $n$, 
\eqref{final} can be reduced to 
$$\frac{i}{(i,d_0)}m\equiv k\left(\text{mod } \frac{d_0}{(i,d_0)}\right).$$
There are at most $\frac{W}{\frac{i}{d_0}}+1$ such choices for $m$.  Therefore,
$$\sum_{\sk{m,n\leq W\\ \ff(m,-n)\equiv 0(d)}}1= \sum_{\sk{m,n\leq W\\(im+gn)^2\equiv 0(d)}}{1}\ll \left(\frac{W}{(i,d_0)}+1\right)\left(\frac{W}{\frac{d_0}{(i,d_0)}}+1\right)\ll W^2d^{-\frac{1}{2}}+W.$$
\end{proof}
Now we can show that  
\begin{lem}
$$\II_Q^{(=,=)}\llet T^{2\delta-1}X^2N^{-\eta}$$
\end{lem} 
 \begin{proof}
 Applying Lemma \ref{innerdsum} and Lemma \ref{finalingrediant} to \eqref{iqeqeq} with $W=\frac{uq_0}{X}$, we have
 \al{\label{eqeq}\nonumber
 \II_Q^{(=,=)}&\lle \frac{N^{\epsilon}UX^4}{QM}\sum_{u<U}\frac{1}{u^4}\sum_{\ff\in\FF}\sum_{Q\leq q\leq 2Q}\sum_{m,n\ll\frac{uq_0}{X}}\frac{(\bff^2,q)}{q}u^4(\ff(m,-n),-8\bff^2)^{\frac{1}{2}}\\
&\nonumber \lle \frac{N^{\epsilon}UX^4}{QM}\sum_{u<U}\sum_{\ff\in\FF}\sum_{Q\leq q\leq 2Q}\frac{(\bff^2,q)}{q}\sum_{m,n\ll\frac{uq_0}{X}}(\ff(m,-n),-2\bff^2)^{\frac{1}{2}}\\
 &\nonumber\lle  \frac{N^{\epsilon}UX^4}{QM}\sum_{u<U}\sum_{\ff\in\FF}\sum_{Q\leq q\leq 2Q}\frac{(\bff^2,q)}{q} \sum_{d_1|-2\bff^2}d_1^{\frac{1}{2}}\sum_{\sk{m,n\ll\frac{uq_0}{X}\\\ff(m,-n)\equiv 0(d_1)}}1\\
 &\nonumber\lle \frac{N^{\epsilon}UX^4}{QM}\sum_{u<U}\sum_{\ff\in\FF}\sum_{Q\leq q\leq 2Q}\frac{(\bff^2,q)}{q} \sum_{d_1|-2\bff^2}d_1^{\frac{1}{2}} \left(\left(\frac{uq_0}{X}\right)^2d_1^{-\frac{1}{2}}+\frac{uq_0}{X}\right)\\
 &\nonumber\lle  \frac{N^{\epsilon}U^4X^4}{QM}\sum_{\ff\in\FF}\sum_{Q\leq q\leq 2Q}\frac{(\bff^2,q)}{q}\cdot\frac{Tq}{X}\\
 &\nonumber\lle  \frac{N^{\epsilon}U^4X^3T}{QM}\sum_{\ff\in\FF}\sum_{d_2|b_{\ff}^2}d_2\sum_{\sk{Q\leq q\leq 2Q\\q\equiv 0(d_2)}}1\\
 &\lle N^{\epsilon}U^4X^2T^{\delta}\lle N^{\epsilon}U^4X^2T^{2\delta-1}T^{1-\delta}
 }
 Therefore, we have a power saving here.
 \end{proof}
 From \eqref{sig1}, \eqref{sig2} and \eqref{eqeq} we obtain
 \begin{lem}\label{i3}
 $$\mathcal{I}_3\llet T^{2\delta-1}X^2N^{-\eta}.$$
 \end{lem}

 \subsection{Proof of Theorem \ref{mainthm} }
 We are now ready to give the proof of Theorem \ref{mainthm} following the strategy at the end of \S 3.1.   
 \begin{proof}[Proof of Theorem \ref{mainthm}]
 
 From Lemma \ref{sdm} we know that 
 $$\sum_{\frac{n}{2}<n<N}|\rn-\rnu(n)|\ll_{\epsilon}\frac{T^{\delta} X^{2+\epsilon}}{U}\llet T^{\delta}X^2N^{-\eta}.$$
 From Lemma \ref{i1}, Lemma \ref{i2} and Lemma \ref{i3} we know that
 $$\sum_{\frac{n}{2}<n<N}|\enu(n)|^2\leq\int_{-1}^{1}|(1-\frak{T}(\theta))\hrnu(\theta)|^2d\theta\llet T^{2\delta-1}X^2N^{-\eta}.$$
 By Cauchy inequality, we then have
 $$\sum_{\frac{n}{2}<n<N}|\enu(n)|\llet T^{\delta}X^2N^{-\eta}.$$
 From Lemma \ref{smm}, we also have
 $$\sum_{n<N}|\mnn-\mnu(n)|\llet T^{\delta}X^2N^{-\eta}.$$
 Since $\mn=\mathcal{R}_N+\en$ and $\mnu=\rnu+\enu$, we then have
 $$\sum_{n<N}|\en(n)-\enu(n)|\llet T^{\delta-1}X^2N^{-\eta}.$$
 As a result,  $$\sum_{n<N}|\en(n)|\llet T^{\delta}X^2N^{-\eta}.$$
 Let $Z$ be the exceptional subset of $\{n|n\equiv \kappa_1(\text{mod }8)\}\cap(\frac{N}{2},N)$ consisting of all numbers which are not represented by our ensemble $\FF$.  Then for $z\in Z$, we have $\mn(z)\gg_{\epsilon} N^{-\epsilon}T^{\delta-1}$.  Since $\mathcal{R}_N(z)=0$, we have $|\en(z)|\gg_{\epsilon} N^{-\epsilon}T^{\delta-1}$.\\
 
 Therefore,  $$|Z|T^{\delta-1}N^{-\epsilon}\lle{\sum_{n\in Z}}|\en(z)|\llet T^{\delta}X^2N^{-\eta}.$$
 So  $|Z|\ll N^{1-\eta}$, and we prove the density one theorem for the $C_1$-orbit under $\Gamma$.   There are six orbits in $\mathcal{P}$, namely $C_1,C_2,C_3,C_{1^{'}},C_{2^{'}},C_{3^{'}}$.  We can prove the same conclusion for every orbit simply by changing the order of components of $\bd{r}$ or $\bd{r^{'}}$.  Thus Theorem \ref{mainthm} follows.  
 \end{proof}
 
 $\bd{Acknowlegement}$
This paper is essentially the content of the author's PhD thesis when he was a graduate student at Stony Brook. The author has a great many thanks to his PhD advisor, Prof. Alex Kontorovich for introducing this beautiful subject to the author and numerous enlightening discussions. The author also thanks the referee for her/his numerous corrections and helpful suggestions when the first edition of this paper was submitted.  In writing up this paper, the author utilizes the codes provided by Prof. Kontorovich for several pictures.  In addition, the author acknowledges support for this work from Prof. Kontorovich's NSF grants DMS-1209373, DMS-1064214, DMS-1001252 and his NSF CAREER grant DMS-1254788.

\bibliographystyle{plain}
\bibliography{Apollonian-3}

\begin{thebibliography}{10}

\bibitem{BF12}
Jean Bourgain and Elena Fuchs.
\newblock A proof of the positive density conjecture for integer {A}pollonian
  circle packings.
\newblock {\em J. Amer. Math. Soc.}, 24(4):945--967, 2011.

\bibitem{BGS10}
Jean Bourgain, Alex Gamburd, and Peter Sarnak.
\newblock Affine linear sieve, expanders, and sum-product.
\newblock {\em Invent. Math.}, 179(3):559--644, 2010.

\bibitem{BGS11}
Jean Bourgain, Alex Gamburd, and Peter Sarnak.
\newblock Generalization of {S}elberg's {$\frac{3}{16}$} theorem and affine
  sieve.
\newblock {\em Acta Math.}, 207(2):255--290, 2011.

\bibitem{BK13}
Jean Bourgain and Alex Kontorovich.
\newblock On the local-global conjecture for integral {A}pollonian gasket.
\newblock {\em Invent. Math.}, July 2013.

\bibitem{Da67}
Harold Davenport.
\newblock {\em Multiplicative number theory}, volume 1966 of {\em Lectures
  given at the University of Michigan, Winter Term}.
\newblock Markham Publishing Co., Chicago, Ill., 1967.

\bibitem{EGJ98}
J.~Elstrodt, F.~Grunewald, and J.~Mennicke.
\newblock {\em Groups acting on hyperbolic space}.
\newblock Springer Monographs in Mathematics. Springer-Verlag, Berlin, 1998.
\newblock Harmonic analysis and number theory.

\bibitem{Fu10}
Elena Fuchs.
\newblock {\em Arithmetic properties of {A}pollonian circle packings}.
\newblock ProQuest LLC, Ann Arbor, MI, 2010.
\newblock Thesis (Ph.D.)--Princeton University.

\bibitem{GV12}
A.~Salehi Golsefidy and P{{\'e}}ter~P. Varj{{\'u}}.
\newblock Expansion in perfect groups.
\newblock {\em Geom. Funct. Anal.}, 22(6):1832--1891, 2012.

\bibitem{GLMWY03}
Ronald~L. Graham, Jeffrey~C. Lagarias, Colin~L. Mallows, Allan~R. Wilks, and
  Catherine~H. Yan.
\newblock Apollonian circle packings: number theory.
\newblock {\em J. Number Theory}, 100(1):1--45, 2003.

\bibitem{GLMWY05}
Ronald~L. Graham, Jeffrey~C. Lagarias, Colin~L. Mallows, Allan~R. Wilks, and
  Catherine~H. Yan.
\newblock Apollonian circle packings: geometry and group theory. {I}. {T}he
  {A}pollonian group.
\newblock {\em Discrete Comput. Geom.}, 34(4):547--585, 2005.

\bibitem{GM08}
Gerhard Guettler and Colin Mallows.
\newblock A generalization of {A}pollonian packing of circles.
\newblock {\em J. Comb.}, 1(1, [ISSN 1097-959X on cover]):1--27, 2010.

\bibitem{IK04}
Henryk Iwaniec and Emmanuel Kowalski.
\newblock {\em Analytic number theory}, volume~53 of {\em American Mathematical
  Society Colloquium Publications}.
\newblock American Mathematical Society, Providence, RI, 2004.

\bibitem{Ki11}
Inkang Kim.
\newblock Counting, mixing and equidistribution of horospheres in geometrically
  finite rank one locally symmetric manifolds.
\newblock 03 2011.

\bibitem{Kl27}
H.~D. Kloosterman.
\newblock On the representation of numbers in the form {$ax^2+by^2+cz^2+dt^2$}.
\newblock {\em Acta Math.}, 49(3-4):407--464, 1927.

\bibitem{KO11}
Alex Kontorovich and Hee Oh.
\newblock Apollonian circle packings and closed horospheres on hyperbolic
  3-manifolds.
\newblock {\em J. Amer. Math. Soc.}, 24(3):603--648, 2011.
\newblock With an appendix by Oh and Nimish Shah.

\bibitem{La67}
D.~G. Larman.
\newblock On the {B}esicovitch dimension of the residual set of arbitrarily
  packed disks in the plane.
\newblock {\em J. London Math. Soc.}, 42:292--302, 1967.

\bibitem{LP82}
Peter~D. Lax and Ralph~S. Phillips.
\newblock The asymptotic distribution of lattice points in {E}uclidean and
  non-{E}uclidean spaces.
\newblock {\em J. Funct. Anal.}, 46(3):280--350, 1982.

\bibitem{MR735226}
C.~R. Matthews, L.~N. Vaserstein, and B.~Weisfeiler.
\newblock Congruence properties of {Z}ariski-dense subgroups. {I}.
\newblock {\em Proc. London Math. Soc. (3)}, 48(3):514--532, 1984.

\bibitem{Pa76}
S.~J. Patterson.
\newblock The limit set of a {F}uchsian group.
\newblock {\em Acta Math.}, 136(3-4):241--273, 1976.

\bibitem{SaLa}
Peter Sarnak.
\newblock Letter to {J}. {L}agarias about integral {A}pollonian packings, June
  2007.

\bibitem{So37}
Soddy.
\newblock The bowl of integers and hexlet.
\newblock {\em Nature}, 139(77-79), 1937.

\bibitem{Su84}
Dennis Sullivan.
\newblock Entropy, {H}ausdorff measures old and new, and limit sets of
  geometrically finite {K}leinian groups.
\newblock {\em Acta Math.}, 153(3-4):259--277, 1984.

\bibitem{Vi13}
Ilya Vinogradov.
\newblock {\em Effective bisector estimate with application to {A}pollonian
  circle packings}.
\newblock ProQuest LLC, Ann Arbor, MI, 2012.
\newblock Thesis (Ph.D.)--Princeton University.

\bibitem{We84}
Boris Weisfeiler.
\newblock Strong approximation for {Z}ariski-dense subgroups of semisimple
  algebraic groups.
\newblock {\em Ann. of Math. (2)}, 120(2):271--315, 1984.

\end{thebibliography}

\end{document}